\documentclass[11pt]{article}

\pdfoutput=1    

\RequirePackage{my_packages}
\usepackage{authblk}

\RequirePackage{my_theorems_en}

\RequirePackage{my_macros}

\RequirePackage{my_tikz}

\usepackage[backend = biber,        
            language = english ,
            style    = numeric-comp ,   
            firstinits = true,
            isbn = false,
            url = false,
            doi = false,
            sorting = nyt,
            backref=true
            ]{biblatex}

\DeclareNameAlias{sortname}{last-first}
\renewbibmacro{in:}{%
  \ifentrytype{article}{}{%
  \printtext{\bibstring{in}\intitlepunct}}}
\makeatletter
\def\blx@maxline{77}
\makeatother

\AtEveryBibitem{%
  \clearfield{note}%
}

\DeclareNameAlias{labelname}{given-family}

\title{Intersection problem for Droms RAAGs}%
\author[1]{Jordi Delgado\thanks{\url{jdelgado@crm.cat}}}
\author[2]{Enric Ventura\thanks{\url{enric.ventura@upc.edu}}}
\author[3]{Alexander Zakharov\thanks{\url{zakhar.sasha@gmail.com}}}

\affil[1,3]{Centre of Mathematics, University of Porto}
\affil[2]{Department de Matem\`atiques, Universitat Polit\`ecnica de Catalunya}
\begin{document}

\maketitle

\begin{abstract}\noindent
    We solve the subgroup intersection problem (\SIP) for any RAAG $G$ of Droms type (\ie with defining graph not containing induced squares or paths of length $3$):
    there is an algorithm which, given finite sets of generators for two subgroups $H,K\leqslant G$, decides whether $H\cap K$ is finitely generated or not, and, in the affirmative case, it computes a set of generators for $H\cap K$. Taking advantage of the recursive characterization of Droms groups, the proof consists in separately showing that the solvability of \SIP\ passes through free products, and through direct products with free-abelian groups.
    We note that most of RAAGs are not Howson, and many (\eg $\Free[2]\times \Free[2]$) even have unsolvable~\SIP.
\end{abstract}

\section{Introduction}

In group theory, the study of intersections of subgroups has been recurrently considered in the literature. Roughly speaking, the problem is \emph{``given subgroups $H,K\leqslant G$, find $H\cap K$''}. However, in the context of Geometric Group Theory, where groups may be infinite, or even non finitely generated, one needs to be more precise about the word \emph{find}, specially if one is interested in the computational point of view.

A group
is said to satisfy \defin{Howson's property} --- or to be \emph{Howson}, for short --- if the intersection of any two (and so, finitely many) finitely generated subgroups is again finitely generated.

Classical examples of Howson
groups include
free-abelian, and free groups. In~$\ZZ^m$ Howson's property is trivial, whereas for free groups it was proved by Howson himself in~\cite{howson_intersection_1954}, where he also gave an algorithm to compute generators for the intersection.

Not far from these groups one can find examples without the Howson property: consider the group $\Free[2]\times \ZZ =\pres{ a,b}{-} \times \pres{ t}{-}$ and the subgroups $H=\langle a,b\rangle$ and $K=\langle ta, b\rangle$; both are clearly 2-generated
but
$
H\cap K
\,=\,
\Set{ w(a,b) \st |w|_a =0 } = \normalcl{b} \leqslant \Free[2] \,,
$
which is not finitely generated. In this context, it is natural to consider the following decision problems.

\medskip
\begin{named}[Subgroup intersection problem, $\SIP(G)$]
Given words $u_1, \ldots, u_n, v_1, \ldots ,v_m$ in the generators of~$G$, decide whether the subgroup intersection $\gen{u_1, \ldots, u_n} \cap \gen{v_1, \ldots ,v_m}$ is finitely generated or not; and, in the affirmative case, compute a generating set for this intersection.
\end{named}


\medskip
\begin{named}[Coset intersection problem, $\CIP(G)$] Given a finite set of words $w,w', u_1, \ldots, u_n, v_1,$ $\ldots ,v_m$ in the generators of $G$, decide whether the coset intersection $w\gen{u_1, \ldots, u_n} \cap w'\gen{v_1, \ldots ,v_m}$ is empty or not; and in the negative case, compute a coset representative.
\end{named}

In~\cite{delgado_algorithmic_2013}, Delgado--Ventura prove that direct products of free-abelian and free groups have both \SIP\ and \CIP\ solvable. The goal of the present paper is to extend
the algebraic arguments given there, in order to achieve similar properties for a much wider family of groups. To this end it is convenient to consider the following variations
for a general finitely presented group~$G$.

\medskip
\begin{named}[Twofold intersection problem, $\TIP(G)$]
Solve both $\SIP(G)$ and $\CIP(G)$.
\end{named}

\medskip
\begin{named}[Extended subgroup intersection problem, $\ESIP(G)$]
Given a finite set of words $u_1, \ldots, u_n,\allowbreak v_1, \ldots ,v_m,w,w'$ in the generators of $G$, decide whether the intersection of the subgroups $H=\gen{u_1, \ldots, u_n}$ and $K=\gen{v_1, \ldots ,v_m}$ is finitely generated or not; and in the affirmative case:
\begin{enumerate*}[ind]
  \item compute a generating set for $H\cap K$, and
  \item decide whether the coset intersection $wH \cap w'K$ is empty or not (denoted by $\CIP_{\fg}$), computing a coset representative in case it is not.
\end{enumerate*}
\end{named}

%

\bigskip

The main result in this paper is about \emph{finitely generated} PC-groups (\aka right-angled Artin groups, or RAAGs).
This prominent class of groups is closely linked to some crucial examples of groups, notably Bestvina and Brady's example of a group which is homologically finite (of type $\mathsf{FP}$) but not geometrically finite (in fact not of type $\mathsf{F_2}$), and Mihailova's example of a group with unsolvable subgroup membership problem.  More recently D.Wise et al. developed a method of showing that a vast amount of groups are virtual subgroups of RAAGs. Wise used this method to solve some well-known problems in group theory, like Baumslag's conjecture on residual finiteness of one-relator groups with torsion. Furthermore, building on the work of Wise and Kahn--Markovic, I. Agol proved the famous virtually fibred conjecture (the last main open problem in 3-manifold theory due to Thurston), by showing that fundamental groups of closed, irreducible, hyperbolic 3-manifolds are virtual subgroups of RAAGs.

\begin{defn}
A group $\Gpi$ is said to be \defin{partially commutative} (a \defin{PC-group}, for short) if it admits a presentation of the form
\begin{equation} \label{eq: pres PC-group}
\Pres{\Geni}{[\geni_i,\geni_j] = \trivial , \text{ whenever } \set{\geni_i,\geni_j} \in \Edgi} \,,
\end{equation}
for some (not necessarily finite) simple graph $\Gri =(\Geni, \Edgi)$. In this case, we say that $\Gpi$ is presented by the \defin{commutation graph} $\Gri$ and write $G=\pcg{\Gri}$. Then, we say that~\eqref{eq: pres PC-group} is a \defin{graphical presentation} for $G$, and $\Geni$ is a \defin{graphical generating set} (or \defin{basis}) for $G$. In the f.g.~case (i.e., when $\Geni$ is finite) we shall refer to $\pcg{\Gri}$ as a \emph{right-angled Artin group} (a \emph{RAAG}, for short).
\end{defn}


A subgraph $\Grii$ of a graph $\Gri = (\Geni,\Edgi)$ is said to be \defin{full} if it has exactly the edges that appear in~$\Gri$ over its vertex set, say $\Genii \subseteq \Geni$; in this case, $\Grii$ is called the \defin{full subgraph of $\Gri$ spanned by $\Genii$} and we write $\Grii \leqslant \Gri$. When none of the graphs belonging to a certain family $\Fami$ appear as a full subgraph of $\Gri$, we say that $\Gri$ is \defin{$\Fami$-free}.

PC-groups can be thought as a family of groups interpolating between two extreme cases:
free-abelian groups (presented by complete graphs), and free groups (presented by edgeless graphs); having as  \emph{graphical generating set} precisely the standard free-abelian
and free bases,
respectively.
More generally,
the PC-group presented by the disjoint union of graphs $\Gri \sqcup \Grii$ is the free product $\pcg{\Gri} *  \pcg{\Grii}$, and the PC-group presented by the join of graphs $\Gri \join \Grii$ (obtained by adding to ${\Gri \sqcup \Grii}$ every edge joining a vertex in $\Gri$ to a vertex in~ $\Grii$) is the direct product $\pcg{\Gri} \times  \pcg{\Grii}$.

Despite the extreme (free and free-abelian) cases being subgroup-closed, this is not the case for PC-groups. Droms characterized the finitely generated PC-groups having this property in the following well known result.

\begin{thm}[\longcite{droms_subgroups_1987}]\label{thm: droms subgroups}
Let $\Gri$ be a finite graph. Then, every subgroup of $\pcg{\Gri}$ is again a (possibly non finitely generated) PC-group if and only if $\Gri$ is~\mbox{$\{\Path[4],\Cycle[4]\}$-free}. \qed
\end{thm}

Here, $\Path[n]$ stands for the \defin{path graph} on $n$ vertices; and $\Cycle[n]$ stands for the \defin{cycle graph} on $n$ vertices.
Accordingly, we say that a graph is a \defin{Droms graph} if it is finite and $\{\Path[4], \Cycle[4]\}$-free, and a PC-group
is a \defin{Droms group} if it is presented by a Droms graph.

\begin{rem}
We note that finite $\{\Path[4], \Cycle[4]\}$-free graphs have received
diverse denominations throughout the literature, including
\emph{comparability graphs of forests} (in~\cite{wolk_comparability_1962}),
\emph{transitive forests} (in~\cite{rodaro_fixed_2013}),
\emph{trivially perfect graphs} (in~\cite{BrandstadtGraphClassesSurvey1987}), and
\emph{quasi-threshold graphs} (in~\cite{jing-ho_quasi-threshold_1996}).
\end{rem}

\subsection{Results}

The main result in the present paper is the following theorem.

\begin{thm}\label{thm: SIP Droms}
Every Droms group has solvable \ESIP\ (and, in particular, solvable \SIP). \qed
\end{thm}

\medskip

The strategy of the proof arises from the following crucial lemma given by Droms on the way of proving~\Cref{thm: droms subgroups}:
\emph{Every nonempty Droms graph is either disconnected, or it contains a central vertex (\ie one vertex adjacent to any other vertex).}

This easily provides the following recursive definition of the Droms family (with both the graphical and the algebraic counterparts):

\begin{cor}[Droms, \cite{droms_subgroups_1987}]\label{cor: recursive Droms graphs}
The family of Droms graphs (\resp Droms groups) can be recursively defined as the smallest family~\DROMSGR\ (\resp $\DROMS$) satisfying the following rules:
\vspace{-10pt}

\noindent\parbox[t]{0.5\textwidth}{\raggedright
\begin{enumerate}[label=\textsf{\textup{[D\arabic*]}}]
\item \label{item: k0 in Droms} $\Kgraph[0]\in \DROMSGR$;
\item \label{item: binary disjoint unions in Droms} $\Gri_{1}, \Gri_{2} \in \DROMSGR \Imp \Gri_{1}  \sqcup \Gri_{2} \in \DROMSGR$;
\item \label{item: sharp cones in Droms} $\Gri \in \DROMSGR \Imp \Kgraph[1] \join \Gri \in \DROMSGR$.
\end{enumerate}}
\noindent\parbox[t]{0.49\textwidth}{\raggedright
\begin{enumerate}[label=\textsf{\textup{[D\arabic*]}}]
\item \label{item: k0 in Droms} $\Trivial\in \DROMS$;
\item \label{item: recursive Droms * Droms} $G_{1}, G_{2} \in \DROMS \Imp G_{1}  * G_{2} \in \DROMS$;
\item \label{item: recursive Z x Droms}$G \in \DROMS \Imp \ZZ \times G \in \DROMS$. \qed
\end{enumerate}}
\end{cor}
Our proof of Theorem~\ref{thm: SIP Droms} is based on the following preservability results for the intersection properties we are interested in.

\begin{thm}\label{thm: SIP-center}
Let $G$ be a Droms group. If $G$ has solvable \SIP, then $\ZZ^m\times G$ also has solvable~\SIP. \qed
\end{thm}

\begin{thm}\label{thm: ESIP-center}
Let $G$ be a Droms group. If $G$ has solvable \ESIP, then $\ZZ^m\times G$ also has solvable~\ESIP. \qed
\end{thm}

\begin{thm}\label{thm: ESIP-free}
If two finitely presented groups $\Gpfi$ and $\Gpfii$ have solvable \ESIP, then their free product $\Gpfi*\Gpfii$ also has solvable~\ESIP. \qed
\end{thm}

\begin{thm}\label{thm: TIP-free}
If two finitely presented groups $\Gpfi$ and $\Gpfii$ have solvable \TIP, then their free product $\Gpfi*\Gpfii$ also has solvable~\TIP. \qed
\end{thm}

To prove~\Cref{thm: SIP-center,thm: ESIP-center} we extend the techniques in~\cite{delgado_algorithmic_2013} from free groups to RAAGs; and to prove~\Cref{thm: ESIP-free,thm: TIP-free} we use Ivanov's techniques to understand and work with subgroups of free products (see~\cite{ivanov_intersection_1999,ivanov_intersecting_2001}). Both are relatively long and technical arguments, each requiring several pages of preliminary considerations. For the sake of clarity, we decided to include them instead of writing a shorter preprint but harder to read.

Our main result (\Cref{thm: SIP Droms}) can be seen as a partial generalization of
\Cref{thm: K-W-M}(i)
by Kapovich--Weidmann--Myasnikov
in the sense that we prove a stronger thesis than them (namely, \ESIP\ instead of \SMP, see~\Cref{fig: dependencies intersection problems}) for a smaller class of groups (Droms instead of coherent PC-groups). In this situation, it is interesting to ask the following questions.

\begin{qst}
Does the group $\pcg{\Path[4]}$ have solvable \SIP? Is it true that a RAAG have solvable \SIP\ if and only if it is Droms?
\end{qst}

Particularly suggestive for us is the result from Aalbersberg--Hoogeboom~\cite{aalbersberg_characterizations_1989} stating that the intersection problem for a partially commutative monoid is solvable if and only if its commutation graph is Droms. The situation is intriguingly similar to that for the~\SMP.

\bigskip

The paper is organized as follows. In Section~\ref{sec: Preliminars} we establish the necessary background and references about Droms groups and algorithmic issues, and we prove the main result, Theorem~\ref{thm: SIP Droms}, modulo Theorems~\ref{thm: SIP-center}, \ref{thm: ESIP-center}, \ref{thm: TIP-free}, and \ref{thm: ESIP-free}. Then, in Section~\ref{sec: Droms direct product case}, we study the direct product case, proving Theorems~\ref{thm: SIP-center} and \ref{thm: ESIP-center}; and finally, in Section~\ref{sec: Droms free product case}, we consider the free product situation proving Theorems~\ref{thm: TIP-free} and~\ref{thm: ESIP-free}. See~\cite[Part III]{delgado_rodriguez_extensions_2017} for a more detailed version of these results.

\section{Preliminaries}\label{sec: Preliminars}

Below we present the necessary preliminaries on algorithmicity and PC-groups.

\subsection{Algorithmic aspects}

Similar preserving properties
concerning free and direct products
were studied for the Membership Problem (\SMP) by \citeauthor{mikhailova_occurrence_1968}. In~\cite{mikhailova_occurrence_1968} she proved that \SMP\ is preserved under free products; whereas in~\cite{mikhailova_occurrence_1958}, she showed that $\Free[2]\times \Free[2]$ has unsolvable membership problem, proving that \SMP\ (and thus \SIP\ and~\CIP) \emph{do not} pass to direct products.

Several obvious relations among the already introduced algorithmic problems are summarized in the diagram below:

\ffigure{
\begin{tikzpicture}[node distance=0.8 and 1.35, on grid,auto]
\node[] (tip) {\TIP};
\node[] (esip) [above right = of tip] {\ESIP};
\node[] (sip) [right = of esip] {\SIP};
\node[] (mp) [below right = of sip] {\SMP};
\node[] (wp) [right = of mp] {\WP};
\node[] (cip) [below right = of tip] {\CIP};
\node[] (cipf) [right = of cip] {$\CIP_{\fg}$};
\path[] (tip) edge[-Implies, double,double distance = 0.6 mm] (esip);
\path[] (esip) edge[-Implies, double,double distance = 0.6 mm] (sip);
\path[] (sip) edge[-Implies, double,double distance = 0.6 mm]
        node[above right] {$(*)$}
        (mp);
\path[] (mp) edge[-Implies, double,double distance = 0.6 mm] (wp);
\path[] (tip) edge[-Implies, double,double distance = 0.6 mm] (cip);
\path[] (cip) edge[-Implies, double,double distance = 0.6 mm] (cipf);
\path[] (cipf) edge[-Implies, double,double distance = 0.6 mm] (mp);
\path[] (esip) edge[-Implies, double,double distance = 0.6 mm] (cipf);
\end{tikzpicture}
}
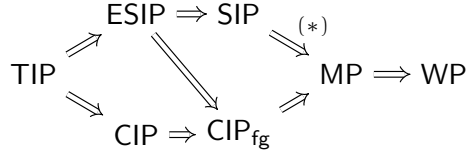
{Some dependencies between algorithmic problems} {fig: dependencies intersection problems}
where the starred implication is true with the extra assumption that the involved group is torsion-free, and has solvable word problem. We recall that \WP\ and \SMP\ stand for the classical \emph{word problem}, and the \emph{subgroup membership problem} stated below.
\medskip
\begin{named}[(Subgroup) membership problem, $\SMP(G)$] Given a finite set of words $w, u_1, \ldots, u_n$ in the generators of $G$, decide whether $w$ represents an element in the subgroup generated by $u_1, \ldots, u_n$; and in the affirmative case compute an expression of $w$ as a word in the~$u_i$'s.
\end{named}

\begin{lem} \label{lem: TF + SIP + WP -> MP}
If a torsion-free group satisfies \SIP\ and \WP, then it also satisfies \SMP.
\end{lem}

\begin{proof}
Let $G=\pres{\Geni}{\Reli}$. Given words $u,v_1,\ldots, v_m$ in $\Geni$, apply \SIP\ to $H=\gen{u}$, and $K=\gen{v_1,\ldots, v_m}$: since $H\cap K$ is cyclic (and so, finitely generated), \SIP\ will always answer \yep, and return a finite set of words $w_1,\ldots,w_{p}$ in $\Geni$ such that $H\cap K=\gen{w_1,\ldots,w_p}=\gen{u^r}$, for some unknown $r\in \ZZ$.

Now, since each $w_i$ must be a power of $u^r$ (say $w_i =u^{r_i}$), we can compute the exponents $r_1,\ldots,r_p \in \ZZ$ by brute force enumeration (even without using WP).
Once we have obtained the integers $r_1,\ldots,r_p \in \ZZ$, we can effectively compute the greatest common divisor $r=\gcd(r_1,\ldots,r_p)$ and get $H\cap K=\gen{w_1,\ldots,w_p} =\gen{u^r}$.

Now, it is clear that $u\in K$ if and only if $u\in H\cap K=\gen{u^r}$; \ie if and only if $u=u^{rs}$, for some $s\in \ZZ$.
To decide whether such an $s$ exists, first apply \WP\ to the input word $u$ in order to decide whether $u=1$ or not. In the affirmative case the answer is obviously \yep; otherwise, $u\neq 1$ and torsion-freeness of $G$ tells us that $u\in K$ (and the answer is \yep) if and only if $r=\pm 1$.
\end{proof}


\begin{rem} \label{rem: CIPfg -> MP}
Note that ${\CIP_{\fg} \Imp \SMP}$ without any further condition, since $\gpi \in \gen{h_1,\ldots,h_k} $ if and only if $ \gpi \cdot \Trivial \cap \trivial \cdot \gen{h_1,\ldots,h_k} \neq \varnothing$.
\end{rem}

\begin{cor} \label{cor: SIP,CIP -> MP (PC-groups)}
For PC-groups, both \SIP\ and \CIP\ imply \SMP. In particular, \SIP\ and \CIP\ are unsolvable for $\Free[2] \times \Free[2] = \Pcg{\Cycle[4]}$, and hence for any PC-group $\Pcg{\Gri}$ with $\Cycle[4]\leqslant \Gri$. \qed
\end{cor}


\begin{rem}
Note that the difference between properties \TIP\ and \ESIP\ is that the second one says nothing about $wH\cap w'K$ in the case when $H\cap K$ is not finitely generated, while \TIP\ is required to answer about emptiness even in this case; this is a subtlety that will become important along the paper.


Finally, note that $(Hw)^{-1}=w^{-1}H$ and $w_1Hw_2 =w_1w_2(w_2^{-1}Hw_2) =w_1w_2H^{w_2}$. Therefore, the variants of $\CIP$ for right, left, and two-sided cosets are equivalent problems.
\end{rem}

\subsection{PC-groups}

Below, we recall some well-known results about PC-groups we will need throughout the paper; we refer the reader to~\cite{charney_introduction_2007,esyp_divisibility_2005,green_graph_1990,koberda_right_2013} for detailed surveys, and further reference.

\begin{thm}\label{thm: PC-group properties}

\begin{enumerate}[ind]
\item\label{thm: Droms isomorphic PC-groups} Let $\Gri_1, \Gri_2$ be simple graphs. Then, the groups~$\pcg{\Gri_1}$, $\pcg{\Gri_2}$ are isomorphic if and only if the graphs $\Gri_1, \Gri_2$ are isomorphic (\cite{droms_isomorphisms_1987}). In particular, the isomorphism problem is solvable within RAAGs.
\item The abelianization of the PC-group $\pcg{\Gri}$ is the free-abelian group of rank $|V\Gri|$.
\item The word and conjugacy problems are solvable for RAAG's (\cite{wrathall_word_1988,wrathall_free_1989,liu_efficient_1990,hermiller_algorithms_1995,van_wyk_graph_1994}).
\item\label{prop: PC-groups are torsion free} PC-groups are torsion-free (\cite{baudisch_subgroups_1981}).
\item\label{prop: RAAGs are Hopfian} RAAGs are residually finite; in particular, they are Hopfian (\cite{koberda_right_2013}).
\item\label{lem: ab rk = max clique} The maximum rank of a free-abelian subgroup of a RAAG $\Pcg{\Gri}$ is the size of a largest clique in~$\Gri$ (\cite{koberda_right_2013}).
\item Disjoint union and graph join correspond, respectively, to free product and direct product of groups: $\pcg{X\sqcup Y}\simeq \pcg{X}*\pcg{Y}$ and $\pcg{X\join Y}\simeq \pcg{X}\times \pcg{Y}$.
\item A PC-group $\pcg{\Gri}$ splits as a nontrivial free product if and only if its defining graph $\Gri$ is disconnected.
\item A PC-group $\pcg{\Gri}$ splits as a nontrivial direct product if and only if its defining graph $\Gri$ is a join.
\item The center of a PC-group $\pcg{\Gri}$ is the (free-abelian) subgroup generated by the set of \central\ vertices in $\Gri$.
\item \label{lem:induced PC-group} Let $\Gri$ be an arbitrary simple graph, and $Y$ a subset of vertices of $\Gri$. Then, the subgroup of $\Pcg{\Gri}$ generated by $Y$ is again a PC-group, presented by the corresponding full subgraph, $\Gen{Y}\simeq \pcg{\Gri[Y]}$. \qed
\end{enumerate}
\end{thm}

Besides Droms groups, other subfamilies of PC-groups naturally arise as directly related with the intersection problem. For example, in~\cite{delgado_characterizations_2014}, Delgado characterized the PC-groups satisfying the Howson property precisely as those being fully residually free (or free products of free-abelian groups).


Another interesting subfamily of PC-groups is that of chordal groups, that is the PC-groups presented by a finite chordal graph (\ie one with no induced cycles of length strictly greater than three).


Clearly,
Droms graphs are chordal and not the other way around. From Theorem~\ref{thm: droms subgroups}, it is clear that Droms groups are \defin{coherent} (every finitely generated subgroup is finitely presented). However, this last class was proved to be bigger, corresponding precisely to chordal groups, which turn out to have some nice algorithmic properties as well.

\begin{thm}[Droms, \cite{droms_graph_1987}]\label{thm: coherent <-> chordal}
Let $\Gri$ be a finite graph. Then, the RAAG $\pcg{\Gri}$ is coherent if and only if $\Gri$ is chordal. \qed
\end{thm}


\begin{thm}[Kapovich--Weidmann--Myasnikov, \cite{kapovich_foldings_2005}]\label{thm: K-W-M}
Let $\Gri$ be a finite chordal graph (\ie $\pcg{\Gri}$ is a coherent RAAG). Then,
\begin{enumerate*}[ind]
\item[(i)] $\pcg{\Gri}$ has solvable membership problem;
\item[(ii)] given a finite subset $\Sseti \subseteq \pcg{\Gri}$, we can algorithmically find a presentation for the subgroup~$\gen{\Sseti} \leqslant \pcg{\Gri}$. \qed
\end{enumerate*}
\end{thm}

We remark the pertinacious absence of $\pcg{\Cycle[4]} = \Free[2] \times \Free[2]$ from any family of algorithmically well behaved groups. However, the exact boundary of the class of RAAGs having solvable~\SMP\ is not known: chordal groups have it, and $\pcg{\Cycle[4]}$ does not. Which of the groups $\pcg{\Cycle[n]}$, for $n\geq 5$, have solvable $\SMP$? Is it possible to find a characterization of the RAAG's with solvable~$\SMP$?

Finally, we recall that for submonoids, the exact border for the corresponding membership problem is already known: In~\cite{lohrey_submonoid_2008} is proved that the submonoid membership problem is solvable in a PC-group if and only if it is Droms. Note that this implies, in particular, that $\pcg{\Path[4]}$ is a group with solvable subgroup membership problem (it is chordal), but unsolvable submonoid membership problem~(it is not Droms).

\subsection{Droms groups}

Due to the recursive description in Corollary~\ref{cor: recursive Droms graphs}, any Droms graph $\Gri$ decomposes as the join of its central part $\Zenter{\Gri}\simeq \Kgraph[m]$, and the full subgraph $\Gri_0=\Gri \setmin \Zenter{\Gri} \leqslant \Gri$, that is
$
\Gri = \Kgraph[m] \join \dGri
$,
where $m\geq 0$ and $\Gri_0$ being either \emph{empty} or a \emph{disconnected} Droms graph; this is called the \emph{primary decomposition} of $\Gri$. In particular:
\begin{enumerate*}[ind]
\item $\pcg{\Gri}$ is free-abelian $\Biimp$
$\Gri$ is complete $\Biimp \Zenter{\Gri} = \Gri \Biimp \dGri = \varnothing$.
\item $\Gri$ is connected $\Biimp$ $\Zenter{\Gri} \neq \varnothing \Biimp m\geq 1\Biimp \Gri$ is a cone.
\item $\pcg{\Gri}$ is centerless $\Biimp \Gri$ is disconnected $\Biimp \Zenter{\Gri} = \varnothing \Biimp m=0$.
\end{enumerate*}


\begin{rem}\label{rem: subgroups of Droms are Droms}
However, all subgroups of Droms groups (including the non finitely generated ones) are again $\{\Path[4],\Cycle[4]\}$-free PC-groups. In particular, every finitely generated subgroup of a Droms group is again a Droms group.
\end{rem}



Finally, we need the following algorithmic result for later use.

\begin{prop}\label{prop: computable bases for Droms groups}
Let $\pcg{\Gri} = \pres{\Geni}{\Reli}$ be a Droms group. Then, there exists an algorithm which, given words $w_1(\Geni),\ldots ,w_p(\Geni)$ in the generators of $\pcg{\Gri}$,
\begin{enumerate*}[ind]
\item computes a basis for the subgroup $H=\gen{w_1,\ldots ,w_p} \leqslant \pcg{\Gri}$;
\item writes the basis elements in terms of the original generators, and vice versa.
\end{enumerate*}
\end{prop}

\begin{proof}
Since Droms graphs are chordal, by~\Cref{thm: K-W-M}, we can effectively compute a finite presentation for $H$, say $H = \pres{\Genii}{\Relii}$.
Then, one can exhaustively explore the tree of all possible Tietze transformations applied to $\pres{\Genii}{\Relii}$ until getting one, say $\pres{\Geniii}{\Reliii}$, in graphical form (namely, with all relators being commutators of certain pairs of generators); this will be achieved in finite time because we know in advance that $H$ is indeed a RAAG.

At this point, we know that $H=\gen{w_1,\ldots ,w_p}\isom \pres{\Genii}{\Relii}\isom \pres{\Geniii}{\Reliii}$, and need to compute expressions for the basis elements $\Geniii = \set{z_1,\ldots,z_r}$ in terms of the $w_i$'s and vice versa. We start a brute force search, using the following two parallel procedures:
\begin{enumerate}[seq]
\item enumerate all homomorphisms $\pres{\Geniii}{\Reliii}\rightarrow H=\gen{w_1,\ldots ,w_p}$; this can be done enumerating all possible $r$-tuples $(v_j )_{j=1}^{r}$ of words in $\set{w_1,\ldots ,w_p}^{\pm}$, and checking whether they determine a well-defined homomorphism.
\item for each such homomorphism $z_j \mapsto v_j$, $j=1,\ldots,r$, analyze whether it is onto $H$ by enumerating all words in $\set{v_1,\ldots,v_r}^{\pm}$, and checking whether each of $w_1, \ldots ,w_p$ appear in the list.
\end{enumerate}
Since we know that $\pres{\Geniii}{\Reliii}\isom H$, there exists a surjective homomorphism from $\pres{\Geniii}{\Reliii}$ onto $H$ and so, the above procedure will eventually find and output one of them, $z_j \mapsto v_j$, $j=1, \ldots ,r$. Finally, since RAAG's are Hopfian (see Theorem~\ref{thm: PC-group properties}\ref{prop: RAAGs are Hopfian}), such a surjective homomorphism is indeed an isomorphism. We then have the $z_j$'s written as words on the $w_i$'s, and the $w_i$'s as words on the $z_j$'s from the stopping criteria at step (2).
\end{proof}

\subsection{Proof of the main result}

The main result in the present paper easily reduces to Theorems~\ref{thm: ESIP-center} and~\ref{thm: ESIP-free}.

\begin{proof}[Proof of \Cref{thm: SIP Droms}.] Let $\Gri$ be a Droms graph, and let $\pcg{\Gri}$ be the corresponding Droms PC-group. We will prove \Cref{thm: SIP Droms} by induction on the number of vertices~$|\Vertexs\Gri|$. If $|\Vertexs\Gri|=0$, then $\pcg{\Gri}=\trivial$, and obviously has solvable \ESIP.

Now, consider a nonempty Droms graph $\Gri$, and assume that every Droms PC-group with strictly less than $|\Vertexs\Gri|$ vertices has solvable \ESIP. Consider the primary decomposition of $\Gri$, say $\Gri=\Kgraph[m]\join \dGri$. If $\Gri_0$ is empty then $\Gri$ is complete, $\pcg{\Gri}\simeq \Z^m$ is free-abelian and so, it has solvable \ESIP. Otherwise, $\Gri_0$ is disconnected, say $\Gri_0 =\Gri_1 \join \Gri_2$ with $\Gri_1$ and $\Gri_2$ being Droms again. By induction, both $\pcg{\Gri_1}$ and $\pcg{\Gri_2}$ have solvable \ESIP, by Theorem~\ref{thm: ESIP-free} $\pcg{\Gri_0}=\pcg{\Gri_1}*\pcg{\Gri_2}$ has solvable \ESIP, and by Theorem~\ref{thm: ESIP-center} $\pcg{\Gri}=\Z^m \times \pcg{\Gri_0}$ also has solvable \ESIP.
\end{proof}


\section{The direct product case} \label{sec: Droms direct product case}

This section is devoted to proving~\Cref{thm: SIP-center,thm: ESIP-center}. To this end, we analyze the Droms groups presented by connected graphs.

\subsection{Preparation}

For all this subsection, we fix an arbitrary connected non-complete Droms graph $\Gri$ and its primary decomposition ${\Gri =\Kgraph[m] \join \Gri_0}$, where $m\geqslant 1$ and $\Gri_0$ is a disconnected Droms graph. Algebraically, ${\pcg{\Gri}=\ZZ^m \times \pcg{\Gri_0}}$, $\Zenter{\pcg{\Gri}}=\ZZ^m$, and $\pcg{\Gri_0}$ is a nontrivial free product. Let ${\Vertexs\Gri_0 = \Geni = \{x_1,\ldots, x_n\}}$ and ${\Vertexs \Kgraph[m] = T = \{ t_1,\ldots ,t_m\}}$.

Every element in $\pcg{\Gri}$ can be written as a word on $\set{ t_1,\ldots ,t_m,x_1, \ldots, x_n }$, where the~$t_i$'s are free to move to any position. We will systematically write all these $t_i$'s ordered on the left, and we will abbreviate them as a vectorial power of a formal symbol `$\mathrm{t}$'. This way, every element in $\pcg{\Gri}=\ZZ^m \times \pcg{\Gri_0}$ can be written in the form
$
t_1^{a_1}\cdots t_m^{a_m}\,u(x_1, \ldots ,x_n) = \faf{a}{u}(x_1,\ldots ,x_n)
$, 
where $\vect{a}=(a_1, \ldots ,a_m)\in \ZZ^m$, and $u=u(x_1,\ldots ,x_n)$ is a word on the $x_i$'s. Clearly, the product of elements is then given by the rule $(\faf{a}{u}) \cdot (\faf{b}{v}) = \faf{a+b}{uv}$.

Quotienting by the center of $\pcg{\Gri}$ gives rise to the short exact sequence
\begin{equation}\label{ses}
 \begin{array}{cccccccccc}
1&\longto &\ZZ^m&\longto  &\pcg{\Gri} & \overset{\prio}\longto &\pcg{\Gri_0}&\longto &1 \, , \\ &&&&\faf{a}{u} &\longmapsto &u
 \end{array}
\end{equation}
where $\prio$ just erases the occurrences of letters in~$T^{\pm} = \{ t_1,\ldots ,t_m\}^{\pm}$.

\begin{defn}
For a given subgroup $H \leqslant  \pcg{\Gri}$, and an element $u\in \pcg{\Gri_0}$, we define the \defin{(abelian) completion} of $u$ in $H$ (the $H$-\emph{completion} of $u$, for short) to be the set
$
\cab{u}{H} = \Set{ \vect{a}\in \ZZ^m \mid \faf{a}{u}\in H } .
$
\end{defn}

\begin{lem}\label{completion}
The completion $\cab{u}{H}$ is either empty (when $u\notin H\prio$), or a coset of~$\ZZ^m \cap H$. More precisely, if $u_1, \ldots ,u_n\in H\prio$, and $\omega (u_1, \ldots ,u_n)$ is an arbitrary word on them, then
$
\cab{\omega(u_1, \ldots ,u_n)}{H} = \sum_{i=1}^{n} \bomega_i \, \cab{u_i}{H}
$, 
where $\bomega_i =|\omega|_i$ is the total exponent of the variable $u_i$ in $\omega$. \qed
\end{lem}



\begin{lem}\label{lem: subgroups of Z^m x disc Droms}
Let $\pcg{\Gri} = \ZZ^m \times \pcg{\Gri_0}$ be the primary decomposition of a connected Droms group. Then, any subgroup $H\leqslant \pcg{\Gri}$ splits as $H=(\ZZ^m \cap H) \times H \prio \sect$, where $\prio \colon \pcg{\Gri} \to \pcg{\dGri}$ is the natural projection killing the center of $\pcg{\Gri}$, and~${\sect \colon H \prio \to H}$ is a section of $\prio{}_{\mid H}$.
\end{lem}

\begin{proof}
Let $\Geni =\set{\geni_1,\ldots,\geni_n}$ be the (finite) set of vertices of $\dGri$ (\ie let $\pcg{\dGri}=\pres{\Geni}{\Reli}$, where $R \subseteq [X,X]$); and let~$\ZZ^m = \pres{t_1 ,\ldots ,t_m}{[t_i,t_j] \ \forall i,j}$. Now, consider the restriction to $H\leqslant \ZZ^m \times \pcg{\dGri}$ of the natural short exact sequence~(\ref{ses}):
\begin{labeledcd}{5}
1 & \longto &\ZZ^m \isom \Zenter{\pcg{\Gri}} & \longto & \pcg{\Gri} & \overset{\prio}\longto & \pcg{\dGri} & \longto & 1 \phantom{\, ,} \\ && \rotatebox[origin=c]{90}{$\leqslant$} && \rotatebox[origin=c]{90}{$\leqslant$} && \rotatebox[origin=c]{90}{$\leqslant$} \nonumber\\
1 & \longto &\ZZ^m \cap \Sgpi  & \longto & \Sgpi & \longto & H \prio & \longto & 1 \, ,                         \label{eq: sses Z^m x disc Droms}
\end{labeledcd}
\noindent
Since $\pcg{\dGri}$ is Droms, we know that $H \prio \leqslant \pcg{\dGri}$ is again a PC-group. Thus, there exists a (not necessarily finite) subset $Y=\set{y_j}_j \subseteq \pcg{\dGri}$ such that $H\prio \isom \pres{Y}{\Relii}$, where $\Relii$ is a certain collection of commutators of the $y_j$'s.

Now, observe that any map $\sect\colon Y\to H$ sending each $y_j \in Y$ back to any of its $\prio$-preimages in $H$ will necessarily respect the relations in $\Relii$: indeed, for each commutator $[y_i,y_j] \in \Relii$, we have $[y_i\sect,y_j\sect]= [t^{\vect{a_i}} y_i,t^{\vect{a_j}}y_j] = [y_i,y_j]$ (for certain abelian completions $\vect{a_i}, \vect{a_j} \in \ZZ^m$). Therefore, any such map $\sect$ defines a (injective) section of the restriction~$\prio{}_{\mid H}$. Thus, the short exact sequence \eqref{eq: sses Z^m x disc Droms} splits and, for any such section $\sect$, $H\prio \isom H\prio \sect \leqslant H$.
Moreover, since the kernel of the subextension \eqref{eq: sses Z^m x disc Droms} lies in the center of $\pcg{\Gri}$, the conjugation action is trivial, and the claimed result follows.
\end{proof}

\begin{cor}\label{cor: H fg <-> H pi fg}
In the above situation, the subgroup $H\leqslant \pcg{\Gri}$ is finitely generated if and only if $H\prio\leqslant \pcg{\Gri_0}$ is finitely generated. \qed
\end{cor}

\begin{rem}
For any $H=\gen{ \mathrm{t}^{\mathbf{b_1} },\ldots ,\mathrm{t}^{\mathbf{b_r} }, \faf{a_1}{u_1}, \ldots , \faf{a_s}{u_s} }\leqslant \pcg{\Gri}=\ZZ^m \times \pcg{\Gri_0}$, where $u_i\neq 1$ for all $i=1,\ldots,s$, we have $\gen{\mathrm{t}^{\mathbf{b_1}}, \ldots ,\mathrm{t}^{\mathbf{b_r}}} \leqslant \ZZ^m\cap H=\zenter{\pcg{\Gri}} \cap H \leqslant \zenter{H}$, but these two inclusions are, in general, not equalities: for the first one, a nontrivial product of the last $s$ generators could, in principle, be equal to $\mathrm{t}^{\mathbf{c}}$ for some element~$\mathbf{c}\not\in \gen{\mathbf{b_1},\, \ldots ,\, \mathbf{b_r}}$; and for the second we could have, for example, $u_1=\cdots =u_s\neq 1$ so that $\faf{a_1}{u_1}$ belongs to $\zenter{H}$ but not to $\zenter{\pcg{\Gri}}$.
\end{rem}

Let us consider now two finitely generated subgroups $\Sgpi_1,\Sgpi_2 \leqslant \pcg{\Gri}$ and analyze when the intersection $\Sgpi_1 \cap \Sgpi_2$ is again finitely generated. We will see that the behaviour of the embedding $(\Sgpi_1 \cap \Sgpi_2)\prio \leqslant (H_1)\prio \cap (H_2)\prio$ is crucial to this end.

\begin{lem}\label{lem: intersection inclusions}
Let $H_1,H_2\leqslant \pcg{\Gri}$. Then,
\begin{enumerate}[ind]
\item $(H_1\cap H_2)\prio \leqslant H_1\prio \cap H_2\prio$, sometimes with strict inclusion;
\item $(H_1\cap H_2)\prio \normaleq H_1\prio \cap H_2\prio$;
\item $[(H_1)\prio \cap (H_2)\prio, (H_1)\prio \cap (H_2)\prio ]\leqslant (H_1 \cap H_2)\prio$.
\end{enumerate}
\end{lem}

\begin{proof}
(i). This is clear.

(ii). To see normality, consider $u\in (H_1 \cap H_2) \prio$, and $v\in (H_1)\prio \cap (H_2) \prio$; then, there exist elements $t^\vect{a}u \in H_1 \cap H_2$, and $t^{\vect{b_i}} v \in H_i$, for $i=1,2$. Now observe that $\fa{a} (v^{-1} u v) = v^{-1} (\faf{a}{u}) v=(\faf{b_i}{v})^{-1}(\faf{a}{u}) (\faf{b_i}{v})\in H_i$, for $i=1,2$. Thus, $\faf{a}{v^{-1} u v} \in H_1 \cap H_2$ and so, $v^{-1} u v \in (H_1 \cap H_2) \prio$.

(iii). Take $ u,v \in (H_1) \prio \cap (H_2)\prio$; then, there exist elements $t^{\vect{a_i}} u \in H_i$, $t^{\vect{b_i}} v \in H_i$, for $i=1,2$. Now, observe that $\comm{u}{v}=u^{-1} v^{-1} uv=(\faf{a_i}{u})^{-1}(\faf{b_i}{v})^{-1} (\faf{a_i}{u}) (\faf{b_i}{v})\in H_i$, for $i=1,2$. Thus, $\comm{u}{v}$ belongs to $H_1 \cap H_2$, and to $(H_1 \cap H_2)\prio$, as claimed.
\end{proof}


\begin{lem}\label{lem: infinitely gen -> infinitely gen}
If $(H_1\cap H_2)\prio$ (and so, $H_1\cap H_2$) is finitely generated, then $(H_1)\prio \cap (H_2)\prio$ is also finitely generated.
\end{lem}

\begin{proof}
Let us assume that $H_1\prio \cap H_2\prio$ is not finitely generated and
find a contradiction.

By~\Cref{rem: subgroups of Droms are Droms}, $H_1\prio \cap H_2\prio$ is again a PC-group with \emph{infinite} $\set{\Path[4],\Cycle[4]}$-free commutation graph, say $\Grii$, and $(H_1\cap H_2)\prio \leqslant \pcg{\Grii'}\leqslant \pcg{\Grii}=H_1\prio \cap H_2\prio$, where $\Grii'$ is the full subgraph of $\Grii$ determined by the vertices appearing in the reduced expressions of elements in $(H_1\cap H_2)\prio$. Note that the assumption of finite generability for $(H_1\cap H_2)\prio$ implies that $\Grii'$ is finite. Note also that, by construction, $\Grii'$ is minimal, \ie for any $x\in \Vertexs\Grii'$, there exists and element $g\in (H_1\cap H_2)\prio$ such that $g\not\in \pcg{\Grii'\setmin \{ x\}}$.

In this situation, $(H_1)\prio \cap (H_2)\prio$ cannot be abelian since, if so, we would have a non finitely generated free-abelian group embedded in the finitely generated PC-group~$\pcg{\Gri}$, which is not possible (see~\Cref{thm: PC-group properties}(vi)). So, $\Delta$ is not complete. Take two non-adjacent vertices, say $u,v$, from $\Delta$; since $\Free[2]\simeq \gen{u,v}\leqslant \pcg{\Grii}$, Lemma~\ref{lem: intersection inclusions}(iii) tells us that $\Free[\infty]\simeq [\gen{u,v}, \gen{u,v}]\leqslant [\pcg{\Grii}, \pcg{\Grii}]\leqslant (H_1 \cap H_2) \prio$ and thus, $(H_1 \cap H_2) \prio$ is not abelian either. Accordingly, neither the infinite graph $\Grii$, nor the finite graph $\Grii'$ are complete.

Suppose now there is a missing edge between some vertex $x\in \Vertexs\Grii'$ and some vertex $y\in \Vertexs\Grii\setmin \Vertexs\Grii'$. Take an element $g\in (H_1 \cap H_2)\prio$ with $g\not\in G(\Grii'\setmin \{x\})$ and Lemma~\ref{lem: intersection inclusions}(ii) would tell us that $y^{-1}g y\in \pcg{\Grii'}$, which is a contradiction.

Hence, in $\Grii$, every vertex from $\Grii'$ is connected to every vertex outside $\Grii'$. But now, take two non-adjacent vertices $\geni_1, \geni_2$ from $\Grii'$
and two non-adjacent vertices $\genii_1, \genii_2$ from $\Grii\setmin \Grii'$
(there must also be some since $\Grii\setmin \Grii'$ is infinite and  $\ZZ^{\infty}$ does not embed into $\pcg{\Gri}$).
Then, the full subgraph of $\Grii$ with vertex set $\{\geni_1, \geni_2, \genii_1, \genii_2\}$ form a copy of $\Cycle[4]$, a contradiction with $\Grii$ being Droms.
\end{proof}

\subsection{Proofs of~\Cref{thm: SIP-center,thm: ESIP-center}}


\begin{proof}[Proof of~\Cref{thm: SIP-center}.]
First of all, observe that we can restrict ourselves to the case where $G$ is a disconnected Droms group.

Therefore, we consider
${\Gri =\Kgraph[m] \join \Gri_0}$ (where $m\geqslant 1$ and $\Gri_0$ is a disconnected Droms graph), we will assume \SIP\ is solvable for $\pcg{\Gri_0}$, and we will prove it solvable for $\pcg{\Gri}=\ZZ^m \times \pcg{\Gri_0}$. Let $\Geni =\Vertexs\Gri_0 =\{x_1,\ldots, x_n\}$ and $T =\Vertexs \Kgraph[m] =\{ t_1,\ldots ,t_m\}$.

Given a finite set of generators for a subgroup, say $H_1\leqslant \pcg{\Gri}$, the first step is to improve them:
project them to $\pcg{\Gri_0}$, and then apply Proposition \ref{prop: computable bases for Droms groups} to compute a basis for $H_1\prio$, say $\{u_1, \ldots ,u_{n_1}\}$ with commutation graph $\pcg{\Grii_1}$. The respective completions, say
$\faf{a_1}{u_1}, \ldots ,\faf{a_{n_1}}{u_{n_1}}\in H_1$
can be computed from the words expressing the $u_i$'s in terms of the projected generators, and recomputing them on the original generators for $H_1$.

Now, for each of the original generators of $H_1$, say $t^c v$, we can write $v\in H_1\pi_0$ in terms of the basis $u_1,\ldots ,u_{n_1}$, say $v=v(u_1, \ldots ,u_{n_1})$ and compute
 $
v(\faf{a_1}{u_1}, \ldots ,\faf{a_{n_1}}{u_{n_1}})
\allowbreak
=
\allowbreak
\fa{d} \,v(u_1, \ldots ,u_{n_1})
\allowbreak
=
\allowbreak
\faf{d}{v} \,.
 $
Since $\faf{c}{v}, \faf{d}{v}\in H_1$, we get $\fa{c-d}\in H_1\cap \ZZ^m$. Repeating this operation for each generator of $H_1$, we get a generating set for $H_1\cap \mathbb{Z}^m$ which is easily reducible to a free-abelian basis, say $\{ \fa{b_1}, \ldots ,\fa{b_{m_1}}\}$.
In this way, we can compute bases for $H_1$ and $H_2$:
 \begin{equation}\label{eq: bases for H1 and H2}
\set{\,\fa{b_1},\ldots ,\fa{b_{m_1}}, \faf{a_1}{u_1}, \ldots , \faf{a_{n_1}}{u_{n_1}}\,}
\text{\, and \,}
\set{\,\fa{b'_1},\ldots ,\fa{b'_{m_2}},\faf{a'_1}{u'_1},\ldots ,\faf{a'_{n_2}}{u'_{n_2}}\,} \,,
 \end{equation}
where $\{ \mathrm{t}^{\mathbf{b_1}},\ldots ,\mathrm{t}^{\mathbf{b_{m_1}}} \}$ and $\{ \mathrm{t}^{\mathbf{b'_1}},\ldots ,\mathrm{t}^{\mathbf{b'_{m_2}}} \}$ are free-abelian bases of $L_1=H_1\cap \ZZ^m$ and $L_2=H_2\cap \ZZ^m$, respectively; and $\{ u_1, \ldots ,u_{n_1}\}$ and $\{ u'_1, \ldots ,u'_{n_2}\}$ are basis of $H_1\pi_0$ and $H_2\pi_0$, with commutation graphs $\Grii_1$, $\Grii_2$.
That is, $H_1\prio \isom \pcg{\Grii_1}$ and $H_2\prio \isom \pcg{\Grii_2}$.

Now, the solvability of \SIP\ in  $\pcg{\Gri_0}$ (assumed by hypothesis) allows us to decide whether $H_1\prio \cap H_2\prio$ is finitely generated or not. If not, then (by Lemma \ref{lem: infinitely gen -> infinitely gen}) neither is $(H_1\cap H_2)\prio$, and we are done.
Thus, 
we can assume that $H_1\prio \cap H_2\prio$ is finitely generated. Then,  the hypothesis provides a finite set of generators
and hence a basis --- say $W=\set{w_1, \ldots ,w_{n_3}}$ with commutation graph $\Grii_3$ --- for $H_1\prio \cap H_2\prio$.
That is,
$
H_1\prio \cap H_2\prio = \pcg{\Grii_3}
=
\gen{w_1, \ldots ,w_{n_3}}
\leqslant
\pcg{\Gri_0} \,,
$
where the $w_i$'s are words on $\Geni$.
After writing each $w_i\in W$ as a word on $U$ and $U'$ respectively --- say $w_i=\omega_i(u_1, \ldots ,u_{n_1})$ and $w_i=\omega'_i(u'_1, \ldots ,u'_{n_2})$ ---
we obtain a description of the inclusions $\moni_1\colon H_1\prio\cap H_2\prio \into H_1\prio$, and $\moni_2\colon H_1\prio\cap H_2\prio \into H_2\prio$ in terms of the corresponding bases.

Abelianizing $\moni_1$ and $\moni_2$, we get the integral matrices $\matr{P_{\!1}}$ (of size $n_3\times n_1$), and $\matr{P_{\!2}}$ (of size $n_3\times n_2$) and complete the upper half of \Cref{fig: esquema interseccio Droms}, where the $\abi_i$'s are the corresponding abelianization maps. Note that, even though $\moni_1$ and $\moni_2$ are injective, their abelianizations $ \matr{P_{\!1}}$ and $ \matr{P_{\!2}}$ need not be ($n_3$ could very well be bigger than $n_1$ or $n_2$).
 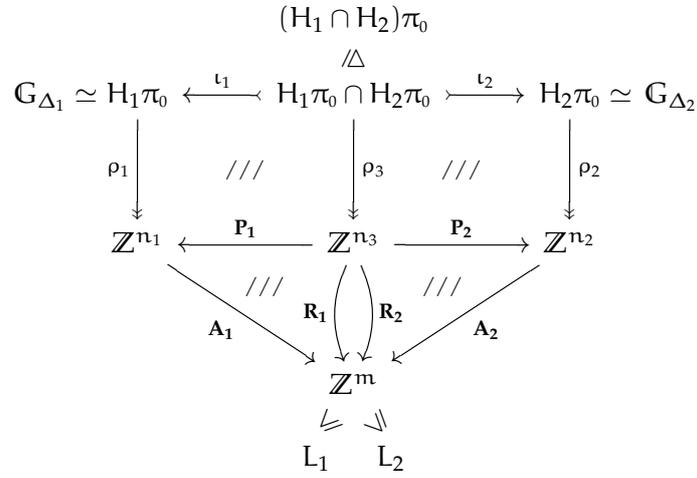
\begin{figure}[h]
 \centering
\begin{tikzcd}[
row sep=37pt, column sep=30pt,
ampersand replacement=\&
]
\&[-37pt]\&(\Sgpi_{1} \cap \Sgpi_{2})\prio\\[-41pt]
\&[-37pt]\&\rotatebox[origin=c]{270}{$\normaleq$}\\[-41pt]
\pcg{\Grii_1} \isom \&[-37pt] \Sgpi_{1} \prio  \arrow[d,two heads,"\abi_{1}"'] \arrow[dr,phantom,"\scriptstyle{///}" description] \& \Sgpi_{1} \prio \cap  \Sgpi_{2} \prio \arrow[d,two heads,"\abi_{3}"]  \arrow[l,tail,"\moni_{1}"'] \arrow[r,tail,"\moni_{2}"] \arrow[dr,phantom,"\scriptstyle{///}" description]\& \Sgpi_{2} \prio \arrow[d,two heads,"\abi_{2}"] \&[-37pt] \isom \, \pcg{\Grii_2} \\
\&[-37pt]\ZZ^{n_1} \arrow[dr,"\matr{A_1}"'] \arrow[drr,phantom,pos=0.23,"\scriptstyle{///}"]\& \ZZ^{n_3} \arrow[l,"\matr{P_1}"'] \arrow[r,"\matr{P_2}"] \arrow[d,bend right=20,"\matr{R_1}"'] \arrow[d,bend left=20,"\matr{R_2}"] \& \ZZ^{n_2} \arrow[dl,"\matr{A_2}"] \arrow[dll,phantom,pos=0.23,"\scriptstyle{///}"]\\
\&[-37pt]{}\& \ZZ^{m}\&{}\\[-41pt]
\&[-37pt]\&\rotatebox[origin=c]{45}{$\leqslant$}  \ \ \rotatebox[origin=c]{135}{$\leqslant$}\\[-43pt]
\&[-37pt]\&L_1 \ \quad \ L_2
\end{tikzcd}
\caption{Intersection diagram for subgroups of Droms groups}\label{fig: esquema interseccio Droms}
 \end{figure}

Now, we can recompute the words $\omega_i$ (\resp $\omega'_i$) as words on the $(\faf{a_i}{u_i})$'s (\resp on the~$(\faf{a'_i}{u'_i})$'s) to get particular preimages of the $w_i$'s in $H_1$ (\resp $H_2$). Namely,
\begin{align*}
&\omega_i (\faf{a_1}{u_1}, \ldots ,\faf{a_{n_1}}{u_{n_1}})
= \mathrm{t}^{\bm{\omega_i} \matr{A_1}}\, {\omega_i(u_1, \ldots ,u_{n_1})}
= \mathrm{t}^{\bm{\omega_i} \matr{A_1}}\,{w_i}
\in H_1 \, ,\\
&\omega'_i (\faf{\mathbf{a'_1}}{u'_1}, \ldots ,\faf{a'_{n_2}}{u'_{n_2}})
= \mathrm{t}^{\bm{\omega'_i} \matr{A_2}}\, {\omega'_i(u'_1, \ldots ,u'_{n_2})}
= \mathrm{t}^{\bm{\omega'_i} \matr{A_2}}\,{w_i}
\in H_2 \, ,
\end{align*}
where $\boldsymbol{\omega_i} = (\omega_i)\ab$, $\boldsymbol{\omega'_i} = (\omega'_i)\ab$; and
$\matr{A_1},\matr{A_2}$
are the integral matrices
(of sizes $n_1\times m$ and $n_2\times m$)
having as rows $\set{\vect{a_1},\ldots,\vect{a_{n_1}}}$ and $\set{\vect{a'_1},\ldots,\vect{a'_{n_2}}}$ respectively.
Hence, the abelian completions of $w_i\in H_1\prio \cap H_2\prio$ in $H_1$ and $H_2$ are the linear varieties:
\begin{align*}
&\cab{w_i}{H_1} = \boldsymbol{\omega_i}\matr{A_1} +L_1 = w_i\moni_1\abi_1 \matr{A_1}+L_1 = w_i \abi_3 \matr{R_1}+L_1 \,,\\ &\cab{w_i}{H_2} = \boldsymbol{\omega'_i}\matr{A_2} +L _2 = w_i\moni_2\abi_2 \matr{A_2}+L_2 = w_i \abi_3 \matr{R_2}+L_2\,,
\end{align*}
where $L_j = \ZZ^m \cap H_j$, and we have used the commutation $\moni_j \abi_j = \abi_3 \matr{P_j}$ together with the definition $\matr{R_j} \coloneqq \matr{P_j} \matr{A_j}$, for $j=1,2$; see \Cref{fig: esquema interseccio Droms}. Note that all maps and matrices involved in \Cref{fig: esquema interseccio Droms} are explicitly computable from the data.

To finish our argument, it suffices to understand which elements of $H_1\prio \cap H_2\prio$ belong to $(H_1\cap H_2)\prio$. They are, precisely, those whose $H_1$-completion and $H_2$-completion intersect:
 \begin{align}
(H_1\cap H_2)\prio
&=\set{\, w\in H_1\prio \cap H_2\prio \mid (w\abi_3  \matr{P_1}\matr{A_1} + L_1) \cap (w\abi_3 \matr{P_2}\matr{A_2}+L_2) \neq \varnothing \,} \notag \\[2pt]
&=\left(\set{ \vect{d}\in \ZZ^{n_3} \mid (\mathbf{dR_1}+L_1)\cap (\mathbf{dR_2}+L_2)\neq \varnothing }\right)\abi_3^{-1} \label{int}\\[2pt]
&=(\set{ \vect{d}\in \ZZ^{n_3} \mid \vect{d} (\matr{R_1}-\matr{R_2}) \in L_1 + L_2 })\abi_3^{-1} \notag\\[2pt]
&=(L_1+L_2)(\mathbf{R_1-R_2})^{-1} \abi_3^{-1}
= M\abi_3^{-1} \, , \notag
 \end{align}
where $M \coloneqq (L_1+L_2)(\matr{R_1}-\matr{R_2})^{-1}$ denotes the full preimage of $L_1+L_2$ by the matrix $\matr{R_1}-\matr{R_2}$ for which a basis is clearly computable using linear algebra. At this point, we can decide whether $(H_1 \cap H_2)\prio$ is finitely generated or not by distinguishing two cases.

If $\Grii_3$ is complete (this includes the case where $\Grii_3$ is empty and $n_3=0$), then~$H_1\prio \cap H_2 \prio \isom \ZZ^{n_3}$ is abelian, $\abi_3$ is the identity, and $(H_1\cap H_2)\prio =(L_1+L_2)(\matr{R_1}-\matr{R_2})^{-1} = M$ is always finitely generated and computable.

So, assume $\Grii_3$ is not complete. Since it is a Droms graph, it will have a primary decomposition, say $\Grii_3 =\Kgraph[n_4] \join \Grii_5$, where $n_4\geqslant 0$, and $\Grii_5$ is Droms again, disconnected, and with $|\Vertexs\Grii_5|= n_5 = n_3 -n_4 \geqslant 2$. Let us rename the vertices $\{ w_1, \ldots ,w_{n_3} \}$ of $\Grii_3$ as $\Vertexs\Kgraph[n_4] \eqqcolon \set{z_1, \ldots ,z_{n_4}}$, and $\Vertexs\Grii_5 \eqqcolon \{y_1, \ldots ,y_{n_5}\}$, depending on whether they belong to $\Kgraph[n_4]$ or $\Grii_5$. This means that $H_1\prio \,\cap\, H_2\prio \isom \pcg{\Grii_3} = \ZZ^{n_4}\times \,\pcg{\Grii_5}$, where $n_4\geqslant 0$, and $\pcg{\Grii_5} \neq 1$ decomposes as a nontrivial free product. Furthermore, the normal subgroup $(H_1\cap H_2)\prio \normaleq \pcg{\Grii_3}$ is \emph{not} contained in $\ZZ^{n_4}$ (taking two vertices, say $y_i, y_j$, in different components of $\Grii_5$,  \Cref{lem: intersection inclusions}(iii) tells us that $1\neq [y_i,y_j]\in (H_1\cap H_2)\prio$).
In this situation, the abelianization map $\abi_3 \colon \pcg{\Grii_3} \onto \ZZ^{n_3}$ is the identity on the center $\ZZ^{n_4}$ of $\pcg{\Grii_3}$ and so, can be decomposed in the form
$\abi_3
=
\operatorname{id}\times \,\abi_5 \colon \pcg{\Grii_3}
=
\ZZ^{n_4}\times \pcg{\Grii_5}  \onto  \ZZ^{n_4}\times \ZZ^{n_5}
=
\ZZ^{n_3}$,
$(\vect{c},v) \mapsto  (\vect{c},\vect{v})$.
where $\vect{v}$ denotes the abelianization $\vect{v} = v\ab \in \ZZ^{n_5}$.
Of course, if $n_4=0$ then~$\abi_5=\abi_3$.

Now consider the image of $(H_1\cap H_2)\prio$ under the projection $\prii\colon \ZZ^{n_4}\times \pcg{\Grii_5} \onto \pcg{\Grii_5}$, which is nontrivial since $(H_1\cap H_2)\prio \not\leqslant \ZZ^{n_4}$. We have $1\neq (H_1\cap H_2)\prio\prii \normaleq \pcg{\Grii_5}$, a nontrivial normal subgroup in a group which decomposes as a nontrivial free product. Therefore,
  \begin{equation}\label{eq: Droms intersection fg conditions}
\begin{aligned}
H_1\cap H_2 \text{ is f.g.}
& \,\Biimp\,  (H_1\cap H_2)\prio\prii \text{ is f.g.}
 \,\Biimp\,  (H_1\cap H_2)\prio\prii \normaleq\fin \pcg{\Grii_5} \\[3pt]
 &\,\Biimp\,  M\abi_3^{-1} \prii \normaleq\fin \pcg{\Grii_5}
 \,\Biimp\,  M{\prii\!\!} \ab \abi_5^{-1} \normaleq\fin \pcg{\Grii_5} \\[3pt]
 &\,\Biimp\,  M{\prii\!\!} \ab \normaleq\fin \ZZ^{n_5}
 \,\Biimp\,  \rk \left(M{\prii\!\!} \ab \right) = n_5 \, .
\end{aligned}
 \end{equation}
The first of these equivalences is a (double) application of Corollary~\ref{cor: H fg <-> H pi fg}. The second one is an application of
the following theorem in \cite[Section~6]{baumslag_intersections_1966} by B.\,Baumslag:
\emph{
Let $G$ be the free product of two nontrivial groups. Let H be a finitely generated subgroup containing a nontrivial normal subgroup of $G$. Then $H$ is of finite index in $G$.}
The fourth equivalence is correct since $\prii\abi_5 =\abi_3 {\prii\!\!} \ab$ and all of them are surjective maps. Finally, the fifth equivalence is correct because following backwards the epimorphism $\abi_5$, a subgroup $M{\prii\!} \ab\leqslant \ZZ^{n_5}$ is of finite index if and only if its full preimage $M{\prii\!} \ab\abi_5^{-1}$ is of finite index in $\pcg{\Grii_5}$, in which case the two indices do coincide, namely,~$\ind{\ZZ^{n_5}}{M{\prii\!\!} \ab}=\ind{\pcg{\Grii_5}}{ M{\prii\!\!} \ab\abi_5^{-1}}$; see Figure~\ref{fig: prii and priiab}.
 \begin{figure}[h]
 \centering
\begin{tikzcd}[row sep=28pt, column sep=28pt, ampersand replacement=\& ]
\hspace{13pt}(H_1 \cap H_2)\prio \,\isom\, M \abi_3^{-1} \, \leqslant \&[-37pt] \pcg{\Grii_3} \arrow[dr,phantom,"\scriptstyle{///}" description] \arrow[r,"\prii"] \arrow[d,two heads,"\rho_3"']\& \pcg{\Grii_5} \arrow[d,two heads,"\rho_5"] \&[-37pt] \geqslant \, M \prii\!\ab \abi_5^{-1} \hspace{30pt}\\
(L_1 + L_2)(\matr{R_1}-\matr{R_2})^{-1} = M \, \leqslant \  \&[-31pt] \ZZ^{n_3} \arrow[r,"\prii\!\ab"]\& \ZZ^{n_5} \&\geqslant \, M \prii\!\ab \hspace{40pt}
\end{tikzcd}
\caption{The map $\prii$ and its abelianization}
\label{fig: prii and priiab}
\end{figure}

Since the map ${\prii\!\!} \ab$ is computable, the last condition in~\eqref{eq: Droms intersection fg conditions} can be effectively checked. Hence, we can algorithmically decide whether $H_1\cap H_2$ is finitely generated or not (ultimately, in terms of some integral matrix having the correct rank). This solves the decision part of $\SIP$.

It only remains to compute a finite set of generators for $H_1\cap H_2$ assuming it is finitely generated, \ie assuming the equivalent conditions in~\eqref{eq: Droms intersection fg conditions} are satisfied.

We first use linear algebra to compute a finite family $C$ of coset representatives of~$\ZZ^{n_5}$ modulo $M{\prii\!\!} \ab$. Then, choose arbitrary $\abi_5$-preimages in $\pcg{\Grii_5}$, say $\{ v_1,\ldots ,v_r\}$, where
$r=\ind{\pcg{\Grii_5}}{M{\prii\!\!} \ab\abi_5^{-1}}=\ind{\ZZ^{n_5}}{ M{\prii\!\!} \ab}$ (we can take, for example, $y_1^{a_1}\cdots y_{n_5}^{a_{n_5}}\in \pcg{\Grii_5}$ for each vector $\vect{a}=(a_1, \ldots ,a_{n_5})\in \ZZ^{n_5}$). Now, construct the Schreier graph of the subgroup
$
(H_1\cap H_2)\prio \prii
=
\allowbreak
M\abi_3^{-1}\prii
=
\allowbreak
M{\prii\!\!} \ab\abi_5^{-1}
\,\leqslant\fin\, \pcg{\Grii_5}
$ 
with respect to $\Vertexs\Grii_5 =\{y_1, \ldots ,y_{n_5}\}$, in the following way: draw as vertices the cosets $\coset{v_1}, \ldots ,\coset{v_r}$; then, for every $\coset{v_i}$ ($i=1,\ldots, r$), and every $y_j$ ($j=1,\ldots ,n_5$), draw an edge labelled $y_j$ from $\coset{v_i}$ to $\coset{v_iy_j}$. Here, we need to algorithmically recognize which is the coset $\coset{v_iy_j}$ from our list of vertices, but this is easy since:
$
\coset{v_iy_j}=\coset{v_k} \,\Biimp\, v_iy_jv_k^{-1}\in M{\prii\!\!} \ab\abi_5^{-1}
\Biimp (v_iy_jv_k^{-1})\abi_5 \in M{\prii\!\!} \ab
$. 

From the Schreier graph of $(H_1\cap H_2)\prio \prii \leqslant\fin \pcg{\Grii_5}$, we can obtain a finite set of generators for $(H_1\cap H_2)\prio \prii$ just reading the labels of the closed  paths $\ATreei[\edgi]$ corresponding to the arcs, $\edgi$, outside a chosen maximal tree $\ATreei$. These will be words on $\Vertexs\Grii_5 =\{y_1, \ldots ,y_{n_5}\}$, \ie elements of $\pcg{\Grii_3}$ not using the central vertices $\{z_1,\ldots, z_{n_4}\}$.

The next step is to lift the obtained generators
to generators of ${(H_1\cap H_2)\prio}$, pulling them back through $\prii$. For each one of them, say~$g(y_1, \ldots, y_{n_5})$, we look for its preimages in $(H_1\cap H_2)\prio$; they all are of the form
$
z_1^{\lambda_1}\cdots z_{n_4}^{\lambda_{n_4}}g(y_1, \ldots ,y_{n_5})\,,
$ 
where the unknowns $\lambda_1, \ldots ,\lambda_{n_4}\in \ZZ$ can be found by solving the system of linear equations coming from the fact $z_1^{\lambda_1}\cdots z_{n_4}^{\lambda_{n_4}}g(y_1, \ldots ,y_{n_5}) \in M\abi_3^{-1}$. That is,
$
(\lambda_1, \ldots, \lambda_{n_4}, |g|_1, \ldots ,|g|_{n_5})(\mathbf{R_1-R_2}) \in L_1+L_2
$. 

For each such $g(y_1, \ldots ,y_{n_5})$, we compute a particular preimage
of the previous form
and put them all, together with a free-abelian basis for
 \begin{equation*}
\ker \prii \cap (H_1\cap H_2)\prio = \Set{ z_1^{\lambda_1}\cdots z_{n_4}^{\lambda_{n_4}} \st (\lambda_1, \ldots ,\lambda_{n_4}, 0,\ldots ,0)(\mathbf{R_1-R_2}) \in L_1+L_2 },
 \end{equation*}
to constitute a set of generators for $(H_1\cap H_2)\prio$.

Finally, we have to lift these generators for $(H_1\cap H_2)\prio$, to a set of generators for $H_1\cap H_2$: for each such generator, say $h_j
$, write it as a word $h_j =\omega_j(u_1, \ldots, u_{n_1})$ and as a word $h_j =\omega'_j(u'_1, \ldots, u'_{n_2})$ in the original bases $U$ for $H_1\prio$ and $U'$ for $H_2\prio$, respectively. Now, reevaluate each $\omega_j$ and $\omega'_j$ in the corresponding basis elements from~\eqref{eq: bases for H1 and H2} for $H_1$ and $H_2$ respectively, to obtain vectors $\vect{c_j},\vect{c'_j} \in \ZZ^m$ such that:
\begin{align*}
&\omega_j( \faf{a_1}{u_1}, \ldots , \faf{a_{n_1}}{u_{n_1}}) = \faf{c_j}{\omega_j(u_1, \ldots, u_{n_1})}=\faf{c_j} h_j \in H_1 \,,\\ &\omega'_j( \faf{a'_1}{u'_1}, \ldots , \faf{a'_{n_2}}{u'_{n_2}}) = \faf{c'_j}{\omega'_j(u'_1, \ldots, u'_{n_2})}= \faf{c'_j} h_j\in H_2 \,.
\end{align*}

Finally, for each $j$, compute a vector $\vect{d_j}\in (\vect{c_j}+L_1)\cap (\vect{c'_j}+L_2)$ (note that these intersections of linear varieties must be nonempty because $h_j\in (H_1\cap H_2)\prio$), and consider the element $\mathrm{t}^{\vect{d_j}}h_j\in H_1\cap H_2$. All these elements $\faf{d_j}{h_j}$, together with a free-abelian basis for~$H_1\cap H_2\cap \ZZ^m =(H_1\cap \ZZ^m )\cap (H_2\cap \ZZ^m )=L_1\cap L_2$ constitute the desired set of generators for $H_1\cap H_2$, and the proof is completed.
\end{proof}



Below, we extend the previous arguments to prove~\Cref{thm: ESIP-center}.

\begin{proof}[Proof of~\Cref{thm: ESIP-center}]
By exactly the same argument as before, we can reduce to a non-complete connected Droms graph $\Gri$ with primary decomposition ${\Gri =\Kgraph[m] \join \Gri_0}$ (where $m\geqslant 1$ and $\Gri_0$ is a disconnected Droms graph), we assume \ESIP\ to be  solvable for $\pcg{\Gri_0}$, and we have to solve it for $\pcg{\Gri}=\ZZ^m \times \pcg{\Gri_0}$.

We are given finite sets of generators for two subgroups $H_1,H_2\leqslant \pcg{\Gri}$, and two extra elements $\faf{a}{u},\, \faf{a'}{u'} \in \pcg{\Gri}$. Since the solvability of \ESIP\ implies that of \SIP, we can apply~\Cref{thm: SIP-center} to effectively decide whether $H_1\cap H_2$ is finitely generated or not, and in the affirmative case compute a basis for $H_1\cap H_2$. We assume all the notation developed along the proof of~\Cref{thm: SIP-center}.

Now, if $H_1\cap H_2$ is not finitely generated there is nothing else to do; otherwise, we can compute a basis, say $\set{v_1,\ldots,v_p}$, for $H_1\cap H_2$ and we have to decide whether the coset intersection $(\faf{a}{u})H_1 \cap (\faf{a'}{u'})H_2$ is empty or not. Note that
$
(\faf{a}{u})H_1 \cap (\faf{a'}{\!u'})H_2 =
\varnothing
$
if and only if
$
((\faf{a}{u})H_1 \cap (\faf{a'}{\!u'})H_2 )\prio
=
\varnothing
$, 
and that
$ 
((\faf{a}{u})H_1 \cap (\faf{a'}{u'})H_2 )\prio \subseteq
\allowbreak
((\faf{a}{u})H_1)\prio \cap ((\faf{a'}{u'})H_2)\prio
=
\allowbreak
u(H_1\prio) \cap u'(H_2\prio)
$. 
Then, since $H_1\cap H_2$ is finitely generated, we know from \Cref{lem: infinitely gen -> infinitely gen} that $H_1\prio\cap H_2\prio$ is finitely generated as well. Hence, an application of the \ESIP\ solvability hypothesis for $\pcg{\Gri_0}$, tells us whether the coset intersection~$u(H_1\prio) \cap u'(H_2\prio)$ is empty or not. If it is empty, then $((\faf{a}{u})H_1 \cap (\mathrm{t}^{\mathbf{a'}}u')H_2)\prio$ is empty as well, and we are done.

Otherwise, $u(H_1\prio) \cap u'(H_2\prio)\neq \varnothing$, and the hypothesis gives us an element $v_0\in u(H_1\prio) \cap u'(H_2\prio)$ as a word on $\Vertexs\Gri_0 =\{ x_1, \ldots ,x_n\}$; further, $u(H_1\prio) \cap u'(H_2\prio)=v_0(H_1\prio \cap H_2\prio)$.

Observe that $((\faf{a}{u})H_1 \cap (\faf{a'}{\!u'})H_2)\prio$ consists precisely of those elements $v_0w$, with $w\in H_1\prio \cap H_2\prio$, for which there exists a vector $\vect{c}\in \ZZ^m$ such that $\faf{c}{v_0 w} \in (\faf{a}{u})H_1 \cap (\faf{a'}{\!u'})H_2 \,;$ that is, such that $\faf{c-a}{u^{-1} v_0 w}\in H_1$, and
$\faf{c-a'}{(u')^{-1} v_0 w}\in H_2$. That is, $\vect{c}-\vect{a}\in \cab{u^{-1} v_0 w}{H_1}$, and $\vect{c}-\vect{a'} \in \cab{(u')^{-1} v_0 w}{H_2}$. Hence, $((\faf{a}{u})H_1 \cap (\mathrm{t}^{\mathbf{a'}} u')H_2) \prio = \varnothing$ if and only if, for all $w\in H_1\prio \cap H_2\prio$,
 $
\left(\vect{a}+{\mathcal C}_{H_1}(u^{-1}v_0w)\right) \allowbreak\cap \left(\mathbf{a'}+{\mathcal C}_{H_2}((u')^{-1}v_0 w)\right) = \varnothing.
 $

Fix an arbitrary word $w=\omega (w_1, \ldots ,w_{n_3})\in H_1\prio \cap H_2\prio$ (for its abelianization, write $|\omega|_i =\lambda_i$, for $i=1,\ldots ,n_3$). Choose vectors $\vect{c}\in \cab{u^{-1} v_0} {H_1}$, $\vect{c'}\in \cab{(u')^{-1} v_0}{H_2}$, and $\vect{d_i}\in \cab{w_i}{H_1}$, $\vect{d'_i}\in \cab{w_i}{H_2}$, for $i=1,\ldots ,n_3$. By Lemma~\ref{completion}:
 \begin{equation*}
\begin{split}
\left(\vect{a} + \cab{u^{-1} v_0 w}{H_1} \right) \,\cap\,& \left( \vect{a'} + \cab{(u')^{-1}v_0w}{H_2} \right) = \\[3pt] &= \left(\vect{a} + \vect{c}+{\mathcal C}_{H_1}(w)\right) \,\cap\, \left(\vect{a'} + \vect{c'}+{\mathcal C}_{H_2}(w)\right)\\[3pt] &= \left(\vect{a} + \vect{c}+\sum_{i=1}^{n_3} \lambda_i {\mathcal C}_{H_1}(w_i)\right) \,\cap\, \left(\vect{a'} + \vect{c'}+\sum_{i=1}^{n_3}\lambda_i {\mathcal C}_{H_2}(w_i)\right) \\[3pt]
&= \left(\vect{a} + \vect{c}+\sum_{i=1}^{n_3}\lambda_i \mathbf{d_i}+L_1\right) \,\cap\, \left(\vect{a'} + \vect{c'}+\sum_{i=1}^{n_3}\lambda_i \mathbf{d'_i}+L_2\right) \,.
\end{split}
 \end{equation*}

Hence, the coset intersection $\left((\faf{a}{u})H_1 \cap (\mathrm{t}^{\vect{a'}}u')H_2\right)\prio$ is empty if and only if for every integer $\lambda_1, \ldots, \lambda_{n_3}\in \ZZ$, $\left(\vect{a-a'+c-c'}\right) \,+\, \sum_{i=1}^{n_3} \lambda_i (\vect{d_i-d'_i}) \notin L_1+L_2$; or equivalently, if and only if
\begin{equation*}
\left( \left( \vect{a-a'+c-c'} \right) \allowbreak+\Gen{\vect{d_1-d'_1},\ldots ,\vect{d_{n_3}-d'_{n_3}}}\right) \cap (L_1+L_2) = \varnothing \,.
\end{equation*}

This can be effectively decided using linear algebra. And furthermore, in case it is not empty, one can compute an explicit element from $(\faf{a}{u})H_1 \cap (\faf{a'}{u'})H_2$, just following the computations back. This completes the proof.
\end{proof}


\section{The free product case} \label{sec: Droms free product case}

In this section, we shall consider the free product case and prove~\Cref{thm: TIP-free,thm: ESIP-free}. We follow the graph-theoretical approach developed by S. Ivanov in~\cite{ivanov_intersection_1999}.

The classical theory of Stallings foldings (see~\cite{stallings_topology_1983}) provides a bijection between the subgroups of a given free group and certain kind of labelled directed graphs (the so-called \emph{Stallings automata}) which when restricted to finitely generated subgroups (corresponding to finite automata) is fully constructive. This geometric approach allowed to solve many algorithmic problems about free groups in a very nice and intuitive way.

S.\,Ivanov (in~\cite{ivanov_intersection_1999}) generalized this machinery to free products. Among other applications, this allowed him to give a modern proof of the Kurosh Subgroup Theorem, and of B.\,Baumslag theorem stating that free products of Howson groups are again Howson
(see~\cite{ivanov_intersection_1999,ivanov_intersecting_2001,ivanov_kurosh_2008}). However, he did not consider algorithmic issues in his approach.

More recently, Kapovich--Weidmann--Myasnikov extended further these folding techniques (see~\cite{kapovich_foldings_2005}) to fundamental groups of graphs of groups. Their method produces the automaton corresponding to a given subgroup, under some conditions on the edge groups (which automatically hold in the case of free products), and leads to the solution of the membership problem in some cases. However, they do not analyze subgroup intersections.

In order to treat algorithmically intersections of subgroups of free groups, we use Ivanov's approach with the necessary technical adaptations to make it fully algorithmic. In the sake of clarity,
we offer here a self-contained exposition.

The idea is to use generalized folding techniques to algorithmically represent any finitely generated subgroup $H\leqslant \Gpfi *\Gpfii$ by a finite graph of certain kind, called a \emph{reduced wedge automaton}, denoted by~$\Ati_{\!H}$. We note that  from such an object $\Ati_{\!H}$ one can already deduce, algorithmically, a Kurosh decomposition for the subgroup $H$.

Note also that following Ivanov's argument it is possible to define $\Ati_{\!H}$ for arbitrary subgroups $H\leqslant \Gpfi *\Gpfii$.
On the other hand, Ivanov gives in~\cite{ivanov_intersection_1999} a generalization of the classical ``pullback'' technique for free groups: given two subgroups $H_1, H_2\leqslant \Gpfi *\Gpfii$, and having at hand corresponding reduced wedge automata $\Ati_{\!H_1}$ and $\Ati_{\!H_2}$, he describes a reduced wedge automaton  ${\junction{H_1}{H_2}}$ for $H_1\cap H_2$, in terms of $\Ati_{\!H_1}$ and $\Ati_{\!H_2}$. Note a substantial difference with the free situation: ${\junction{H_1}{H_2}}$ may very well be an infinite object even with $\Ati_{\!H_1}$ and $\Ati_{\!H_2}$ being finite (corresponding to the possible non-Howson situation; \ie~$H_1\cap H_2$ may very well be non finitely generated, even with $H_1$ and $H_2$ being finitely generated).

The main argument in the present section is the following: given finite generating sets for $H_1$ and $H_2$, we are able to construct $\Ati_{\!H_1}$ and $\Ati_{\!H_2}$, and then start constructing ${\junction{H_1}{H_2}}$ (even with the possibility of this being infinite). The crucial point is that, in finite time while the construction is running, we are able to either detect that $H_1\cap H_2$ is not finitely generated, or to complete the construction of ${\junction{H_1}{H_2}}$; in the first case we have algorithmically deduced that the intersection $H_1\cap H_2$ is not finitely generated, and in the second case we have effectively constructed ${\junction{H_1}{H_2}}$, from which we shall be able to extract a finite set of generators for $H_1\cap H_2$. In order to check whether $H_1\cap H_2$ is not finitely generated the hypothesis of \SIP\ or \ESIP\ in the factor groups $\Gpfi$ and $\Gpfii$ will be crucial.

\subsection{Wedge automata}

We assume the reader familiar with standard Stallings automata (representing subgroups of free groups, say $\Free[2]=\langle a\rangle * \langle b\rangle$, as involutive $\set{a,b}$-automata recognizing exactly the elements in the corresponding subgroup).


To cover the more general situation of $\Gpfi*\Gpfii$, we need to encode more information into the arcs. A classical $a$-labelled arc would correspond to what we call here a $G_1$-wedge: an arc subdivided in two halves, admitting a (possibly trivial) label from $\Gpfi$ on each side, and also admitting a (possibly trivial) subgroup $A \leqslant \Gpfi$ as a label of the middle (special) vertex between the two halves. Doing the same with the $b$-arcs (and subgroups of $\Gpfii$) we get an automaton with two types of vertices, \emph{primary} (the original ones), and \emph{secondary} (the new ones). See~\Cref{fig: wedges}.

\ffigure{
\begin{tikzpicture}[shorten >=1pt, node distance=1 and 1.3, on grid,auto,>=stealth']
\node[state] (p0) {};
\node[sstate1] (q1) [above right = of p0]{$\scriptstyle{A}$};
\node[] (l1) [below = 1.5 of q1] {$\scriptstyle{(g_1,g'_1\in G_1, \ A \leqslant G_1)}$};
\node[state] (p1) [below right = of q1]{};
\node[] (q3) [right = of p1]{};
\node[state] (p4) [right = of q3]{};
\node[sstate2] (q5) [above right = of p4]{$\scriptstyle{B}$};
\node[] (l2) [below = 1.5 of q5] {$\scriptstyle{(g_2,g'_2\in G_2, \ B \leqslant G_2)}$};
\node[state] (p5) [below right = of q5]{};
\path[->,red](p0) edge[] node[above left = -0.05,swap] {\scriptsize{$\gpi_1$}}(q1);
\path[->,red](q1) edge[] node[above right= -0.03] {\scriptsize{$\gpi'_1$}}(p1);
\path[->,blue](p4) edge[] node[above left= -0.05,swap] {\scriptsize{$\gpi_2$}}(q5);
\path[->,blue](q5) edge[] node[above right= -0.03] {\scriptsize{$\gpi'_2$}}(p5);
\end{tikzpicture}
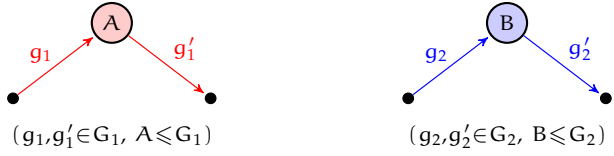
}{Wedges of first and second kind} {fig: wedges}

In these new automata, \walk s are going to spell \emph{subsets} of $\Gpfi*\Gpfii$ (instead of \emph{words} in $\set{a,b}^{\pm}$), by picking \emph{all} the elements from the label of a secondary vertex when traversing it. Allowing, in addition, vertices to have any degree, we get the new notion of wedge automaton.

\begin{defn} \label{def: wedge automaton}
Let $\Gpfi, \Gpfii$ be two arbitrary groups. A \defin{(wedge) $(G_1,G_2)$-automaton} is a septuple $\Ati =(\Vertexs\Ati, \Edgi \Ati, \start, \term, \labl,^{-1}, \bp)$, where:
\begin{enumerate}[ind]
\item $\Gri =(\Vertexs\Ati, \Edgi\Ati, \start, \term,{}^{-1})$ is an involutive digraph (called \defin{underlying digraph} of $\Ati$) with three disjoint types of vertices, $\Vertexs\Ati =\Vertexs_{0}\Ati \sqcup \Vertexs_1\Ati \sqcup \Vertexs_2\Ati$; namely, \defin{primary} (those in $\Vertexs_0 \Ati$, denoted by~\pv), \defin{1-secondary} (those in $\Vertexs_1 \Ati$, denoted by~$\svi$, and \emph{$2$-secondary} (those in $\Vertexs_2 \Ati$, denoted by~$\svii$); and with all arcs in $\Edgi \Ati$ joining (in either direction) a primary vertex with a secondary one;
\ie $\Edgi\Ati =\Edgi_1\Ati \sqcup \Edgi_2\Ati$, where the arcs in $\Edgi_j$ are called \defin{$j$-arcs} and connect primary vertices with $j$-secondary vertices (for $j=1,2$).
\item $\labl$ is a twofold \defin{label map}: for $\nu=1,2$,
 $\labl\colon \Edgi_{\nu}\Ati \to G_{\nu}$, $\edgi \mapsto \labl_\edgi$ compatible with the involution in $\Gri$; and
 $\labl\colon \Vertexs_{\nu}\Ati \to \textsf{Sgp}(G_{\nu})$, $\vertii \mapsto \labl_\vertii$.

\item $\bp$ is a distinguished primary vertex called the \emph{basepoint} of $\Ati$.
\end{enumerate}
\end{defn}

We say that a wedge automaton $\Ati$ is \defin{connected} (\resp \defin{finite}) if the underlying undirected graph is so; note that, by definition, it is always primary-secondary bipartite. We will also say that a wedge automaton $\Ati$ is of \defin{finite type} if the underlying digraph is finite, and the subgroups labelling the secondary vertices are all finitely generated. This will always be the situation when we consider computational issues; in this case, the labels of vertices will usually be given by finite sets of generators. We say that a vertex label is \defin{trivial} if $\labl_\vertii= \Trivial$. If not stated otherwise all the wedge automata appearing from this point will be assumed to be finite.

Recall that
wedge automata are involutive (as automata): for every arc $\edgi \equiv \verti \arc{} \vertii$ reading $\gpi$, there exists a unique inverse arc $\edgi^{-1} \equiv \vertii \arc{} \verti$ reading $\gpi^{-1}$. Hence, a wedge automata $\Ati$ can always be represented by one of its (say positive, denoted by $\Edgi^{+}\Ati$) arc orientations. Then $\Edgi\Ati = \Edgi^{\pm}\Ati \coloneqq \Edgi^{+}\Ati \sqcup \Edgi^{-}\Ati$, where $\Edgi^{-}\Ati$ is the set of inverses of the arcs in~$\Edgi^{+}\Ati$.
A \defin{\walk} in $\Ati$ is a sequence of alternating and successively incident vertices and arcs, starting and ending at primary vertices, $\walki = \verti_0 (\edgi_1^{-1} \vertii_1 \edgi'_1) \verti_1 (\edgi_2^{-1} \vertii_2 \edgi'_2) \verti_2 \cdots \verti_{r-1} (\edgi_r^{-1} \vertii_r \edgi'_r) \verti_r$,  where $\verti_0, \ldots ,\verti_{r}$ are (not necessarily distinct) primary vertices, $\vertii_1, \ldots ,\vertii_r$ are (not necessarily distinct) secondary vertices, and for every $k=1,\ldots ,r$, all three of ${\edgi_k,\vertii_k,\edgi_k'}$ are simultaneously of the same type $\nu_k$ ($\nu_k=1,2$).
A \defin{$\nu$-elementary \walk} is a \walk\ of length $2$ visiting a secondary vertex of $\nu$-type, $\nu=1,2$; it is \defin{degenerate} if it consists of two mutually inverse arcs; otherwise it is called \emph{non-degenerate}. Every \walk\ $\walki$ decomposes as a product of elementary \walk s (either degenerate or nondegenerate, and with possible repetitions) in a unique way, corresponding to the brackets in the expression above: this is called the \defin{elementary decomposition} of $\walki$ (for convention, we take $r=0$ when the \walk\ $\walki$ is trivial). We say that a \walk\ $\walki$ is \defin{alternating} if its elementary decomposition sequence $\walki_1, \walki_2, \ldots ,\walki_r$ alternates between types $1$ and $2$.

The \defin{length} of a \walk\ is the number of arcs in the sequence defining it, \ie twice the number of elementary walks $r$ in its elementary decomposition.

\begin{rem} \label{rem: no backtracking}
Note that $\walki$ involves no backtracking if and only if the $\walki_i$'s in its elementary decomposition $\walki=\walki_1 \cdot \walki_2 \cdots \walki_r$ are all nondegenerate, and there is no backtracking in the consecutive products $\walki_i\cdot \walki_{i+1}$.
\end{rem}

\begin{defn}
The \defin{label} of a walk $\walki =\verti_0 (\edgi_1^{-1} \vertii_1 \edgi'_1) \verti_1 (\edgi_2^{-1} \vertii_2 \edgi'_2) \verti_2 \cdots \verti_{r-1} (\edgi_r^{-1} \vertii_r \edgi'_r) \verti_r$, denoted by $\labl_{\walki}$, is the subset
$\labl_{\walki}
=
(\labl_{\edgi_1}^{-1} \labl_{\vertii_1} \labl_{\edgi'_1})
(\labl_{\edgi_2}^{-1} \labl_{\vertii_2} \labl_{\edgi'_2})
 \cdots
(\labl_{\edgi_r}^{-1} \labl_{\vertii_r} \labl_{\edgi'_r})
\subseteq \Gpfi*\Gpfii$.
That is, while travelling along $\walki$, when we traverse an arc $\edgi$, we \emph{pick} its label $\labl_\edgi$, and when we traverse a secondary vertex~$\vertii$ we take all labels $c \in \labl_\vertii$ (primary vertices have no contribution to $\labl_{\walki}$). Picking always the trivial element when visiting a secondary vertex, we obtain the so-called \defin{basic label} of $\walki$, $\labl^{\bullet}_{\walki}= (\labl_{\edgi_1}^{-1} \labl_{\edgi'_1}) (\labl_{\edgi_2}^{-1} \labl_{\edgi'_2}) \cdots (\labl_{\edgi_r}^{-1} \labl_{\edgi'_r}) \in \labl_{\walki} \subseteq \Gpfi*\Gpfii$.
\end{defn}

It is clear that $\labl_{\alpha^{-1}} = \labl_{\alpha }^{-1}$ and $\labl_{\alpha}\cdot \labl_{\beta}= \labl_{\alpha\beta}$.

\begin{rem}\label{sylabes}
Note also that if
$\walki =\verti_0 (\edgi_1^{-1} \vertii_1 \edgi'_1) \verti_1
\allowbreak
(\edgi_2^{-1} \vertii_2 \edgi'_2) \verti_2
\allowbreak
\cdots \verti_{r-1} (\edgi_r^{-1} \vertii_r \edgi'_r) \verti_r$
is alternating, and we take elements $c_i \in \labl_{\vertii_i}$ such that $\labl_{\edgi_i}^{-1} c_i \labl_{\edgi'_i}\neq 1$ for all $i=1,\ldots ,r$, then the brackets in the expression $(\labl_{\edgi_1}^{-1} c_1 \labl_{\edgi'_1}) (\labl_{\edgi_2}^{-1} c_2 \labl_{\edgi'_2}) \cdots (\labl_{\edgi_r}^{-1} c_r \labl_{\edgi'_r})$ indicate, precisely, the syllable decomposition of the element in $\Gpfi*\Gpfii$ read by $\walki$; otherwise, some consecutive pairs of brackets may merge into the same syllable. This remark will be crucial later.
\end{rem}

Let $\Ati$ be a \wedged\ and let $\verti,\verti'$ be two primary vertices. We define the \defin{coset recognized} by~$\Ati$ relative to $(\verti,\verti')$ to be the set $\recg{\Ati}_{(\verti,\verti')} \coloneqq \cup _{\walki} \, \labl_{\walki}$, where the union runs over all \walk s in $\Ati$ from $\verti$ to $\verti'$. When $\verti=\verti'$, then we abbreviate $\recg{\Ati}_{\verti} :=\recg{\Ati}_{(\verti, \verti)}$. Moreover, if $\verti = \verti' = \bp$
then we simply write $\recg{\Ati} = \recg{\Ati}_{\sbp}$, and we call it the  \defin{subgroup recognized} by  $\Ati$. The lemma below, which is straightforward to prove, justifies this terminology.

\begin{lem}\label{lem: wedge automaton recognition}
Let $\Ati$ be a $(\Gpfi,\Gpfii)$-wedge automaton, and let $\verti,\verti'\in \Vertexs_{0}\,\Ati$. Then,
\begin{enumerate}[ind]
\item $\recg{\Ati}_{\verti}$ is a subgroup of $\Gpfi*\Gpfii$;
\item $\recg{\Ati}_{\verti}$ and $\recg{\Ati}_{\verti'}$ are conjugate to each other; viz.~$\recg{\Ati}_{\verti'} =(\recg{\Ati}_{\verti})^{g}$, for every $g\in \recg{\Ati}_{(\verti,\verti')}$;
\item $\recg{\Ati}_{(\verti,\verti')}$ is a right coset of $\recg{\Ati}_{(\verti,\verti)}$; viz.~$\recg{\Ati}_{(\verti,\verti')} =\recg{\Ati}_{(\verti,\verti)} \cdot g$, for every $g\in \recg{\Ati}_{(\verti, \verti')}$. \qed
\end{enumerate}
\end{lem}


\begin{prop}\label{existence}
For every subgroup $H\leqslant \Gpfi*\Gpfii$, there exists a \wedged\ $\Ati$ recognizing $H$. Furthermore, if $H$ is finitely generated, one such $\Ati$ is of finite type and algorithmically constructible from a finite set of generators for $H$ given in normal form.
\end{prop}

\begin{proof}
Let $H=\langle W \rangle$, where $W=\{ w_1, w_2, \ldots\}$ is a set of generators for $H$. For every nontrivial generator in $W$, say $w$, consider its normal form as an element of~$\Gpfi*\Gpfii$, say $w=a_1 b_1\cdots a_s b_s$, with $s\geqslant 1$, $a_i\in \Gpfi$, $b_i\in \Gpfii$, $a_i\neq 1$ for all $i=2,\ldots ,s$, and $b_i\neq 1$ for all ${i=1,\ldots, s-1}$. Let~$\Flower{w}$ denote the \wedged\ depicted in Fig.~\ref{fig: wedge petal}, and called the \defin{petal automaton} corresponding to $w$. Clearly, $\recg{\,\Flower{w}}=\langle w \,\rangle$.

\ffigure{
\begin{tikzpicture}[shorten >=1pt, node distance=0.4 and 1.3, on grid,auto,>=stealth']
\node[state, accepting] (0) {};
\node[sstate] (a) [above right =  of 0] {$\scriptstyle{\trivial}$};
\node[state] (1) [above right =  of a] {};
\node[sstate] (b) [right =  of 1] {$\scriptstyle{\trivial}$};
\node[state] (2) [right =  of b] {};
\node[sstate] (d) [below right =  of 0] {$\scriptstyle{\trivial}$};
\node[state] (4) [below right =  of d] {};
\node[sstate] (c) [right =  of 4] {$\scriptstyle{\trivial}$};
\node[state] (3) [right =  of c] {};
\path[->](a) edge[] node[above left= -0.05] {$\trivial$}(0);
\path[->](a) edge[] node[above left] {\scriptsize{$a_1$}}(1);
\path[->](b) edge[] node[above] {$\trivial$}(1);
\path[->](b) edge[] node[above] {$b_1$}(2);
\path[->,dashed] (2) edge[bend left,out=90,in=90,min distance=10mm] node[right,] {\scriptsize{$a_2\cdots b_{s-1}$}}(3);
\path[->](c) edge[] node[below] {\scriptsize{$\trivial$}}(3);
\path[->](c) edge[] node[midway,below] {$a_s$}(4);
\path[->](d) edge[] node[below] {\scriptsize{$\trivial$}}(4);
\path[->](d) edge[] node[below] {\scriptsize{$b_s$}}(0);
\end{tikzpicture}
}{A wedge petal}{fig: wedge petal}

Now consider $\Flower{W}$ the disjoint union of all the $\Flower{w_i}$'s identifying the basepoints into a single primary vertex (declared as basepoint); the resulting object is a \wedged\  called the \defin{flower automaton} corresponding to $W$. Clearly, $\recg{\Flower{W}} =\langle W\rangle =H$. Moreover, if $|W|<\infty$ then $\Flower{W}$ is of finite type and constructible.
\end{proof}


\subsection{Reduced wedge automata}

In the same vein as in the classical Stallings graphs, we will ask our wedge graphs to be `deterministic' (in a precise sense specified below). Similar constructions are called `irreducible graphs' by Ivanov in~\cite{ivanov_intersection_1999}, and are particular cases of the so-called `folded graphs' in Kapovich--Weidman--Miasnikov~\cite{kapovich_foldings_2005}.

\begin{defn}\label{def: reduced wedge automata}
Let $\Gpfi, \Gpfii$ be two groups, and let $\Ati$ be a finite \wedged. We say that $\Ati$ is \defin{reduced} if the following conditions are satisfied:
\begin{enumerate*}[ind]
\item \label{item: reduced connected} $\Ati$ is connected;
\item \label{item: reduced deterministic}
every primary vertex of $\Ati$ is incident with at most one arc from $\Edgi_1\Ati$, and at most one arc from $\Edgi_2\Ati$;
\item \label{item: reduced wedge 1-free} no nondegenerate elementary \walk\ reads the trivial element; that is, for $\nu=1,2$, every $\nu$-secondary vertex $\vertii\in \Vertexs_{\nu}\,\Ati$, and every pair of \emph{different} \mbox{$\nu$-arcs} $\edgi_1, \edgi_2$ with $\start \edgi_1=\start \edgi_2=\vertii$, we have that $1\not\in \labl_{\edgi_1}^{-1} \labl_\vertii \labl_{\edgi_2}$ (equivalently, ${\labl_{\edgi_1}\labl_{\edgi_2}^{-1} \not\in \labl_\vertii}$).
\end{enumerate*}
\end{defn}

In a reduced wedge automaton $\Ati$, \cref{rem: no backtracking} can be restated in the following way: $\walki$ presents no backtracking if and only if it is alternating and the elementary \walk s on its elementary decomposition are all nondegenerate. In this case, additionally, property~(iii) from~\Cref{def: reduced wedge automata} ensures that the elementary decomposition of $\walki$ gives the syllable decomposition of every $g\in \labl_{\walki}$ as element from $\Gpfi*\Gpfii$ (since nondegenerate elementary \walk s do not admit the trivial element as a label). However, this is not the whole story: even with some of the $\walki_i$'s being degenerate, we can still get the syllable decomposition of $g\in \labl_{\walki}$ assuming that the elements picked from the labels of the backtracking vertices (if any) are nontrivial. This motivates the following lemma, definition and the subsequent important technical lemma.

\begin{lem}\label{no1}
For any reduced $(\Gpfi, \Gpfii)$-automaton $\Ati$, and any nontrivial walk $\gamma$ between two primary vertices and without backtracking, $1\not\in \labl_{\gamma}$.
\end{lem}

\begin{proof}
This is clear since, by the previous remark, the label of $\walki$ gives its syllable decomposition as element of $G_1 *G_2$.
\end{proof}

\begin{defn}
Let $\walki$ be a \walk\ in a wedge automaton $\Gri$, with elementary decomposition
$\walki =\walki_1 \, \walki_2 \,\cdots\, \walki_r =\verti_0 (\edgi_1^{-1} \vertii_1 \edgi'_1) \verti_1 (\edgi_2^{-1} \vertii_2 \edgi'_2) \verti_2 \cdots \verti_{r-1}(\edgi_r^{-1} \vertii_r \edgi'_r) \verti_r$.
We define the \defin{reduced label} of $\walki$ as
\begin{equation*}
\red{\labl}_{\walki }=\Set{ (\labl_{\edgi_1}^{-1} c_1 \labl_{\edgi'_1}) (\labl_{\edgi_2}^{-1} c_2 \labl_{\edgi'_2}) \cdots (\labl_{\edgi_r}^{-1} c_r \labl_{\edgi'_r}) \relmiddle{|} \begin{array}{l} c_i\in \labl_{\vertii_i} \\ c_i\neq 1 \text{ if } \walki_i \text{ is degenerate} \end{array} \!\!}
\,\subseteq\,
\labl_{\walki}.
\end{equation*}

\end{defn}

\begin{lem}\label{alternating}
For a reduced $(\Gpfi, \Gpfii)$-automaton $\Ati$, we have $\recg{\Ati} =\cup_{\walki} \labl_{\walki }=\cup_{\,\what{\walki}} \,\red{\labl}_{\what{\walki}}$, where the first union runs over all \bp-\walk s $\walki$ of $\Ati$, and the second one only over the alternating \bp-\walk s $\what{\walki}$ of $\Ati$.
\end{lem}

\begin{proof}
The inclusion `$\supseteq$' is clear, since the first union is over more sets than the second one, and $\labl_{\walki}\supseteq \red{\labl}_{\walki}$.
To see `$\subseteq$', fix a \bp-\walk\ $\walki =\walki_1\cdots \walki_r$, and an element $g\in \labl_{\walki}$, and let us find an alternating \bp-\walk\ $\what{\walki}$ such that $g\in \red{\labl}_{\what{\walki}}$.

In fact, if $\walki$ is not alternating then there are $\nu=1,2$ and $i=1,\ldots ,r-1$ such that $\walki_i=\verti_{i-1}\edgi_i^{-1} \vertii_i \edgi'_i \verti_i$ and $\walki_{i+1}=\verti_{i}\edgi_{i+1}^{-1} \vertii_{i+1} \edgi'_{i+1} \verti_{i+1}$ are both of type $\nu$; so, by condition (ii) in~\Cref{def: reduced wedge automata}, $\vertii_i=\vertii_{i+1}$ and $\edgi'_i=\edgi_{i+1}$. Replacing $\walki_i\walki_{i+1}$ by $=\verti_{i-1}\edgi_i^{-1} \vertii_i \edgi'_{i+1} \verti_{i+1}$, we get a new \bp-\walk\ $\walki_{\!_{(1)}}$, with shorter elementary decomposition and such that $g\in \labl_{\walki_{\!_{(1)}}}$ as well, since $(\labl_{\edgi_i}^{-1} c_i \labl_{\edgi'_i}) (\labl_{\edgi_{i+1}}^{-1} c_{i+1} \labl_{\edgi'_{i+1}}) =\labl_{\edgi_i}^{-1} (c_i c_{i+1})\labl_{\edgi'_{i+1}}$, for all $c_i, c_{i+1}\in \labl_{\vertii_i}\leqslant G_{\nu}$.
Repeating this operation a finite number of times, say $k$, we can assume that~${\what{\walki} = \walki_{\!_{(k)}}}$ is alternating and $g\in \labl_{\what{\walki}}$.


It remains to prove that, maybe simplifying $\hat{\gamma}$, $g\in \tilde{\ell}_{\hat{\gamma}}$: if $\vertii_j$ is the secondary vertex in the degenerate elementary \walk\ $\walki_j =\verti_{j-1}\edgi_j^{-1}\vertii_j\edgi_j \verti_{j-1}$, and the corresponding $c_j$ picked in the formation of $g$ is trivial, then just ignore $(\labl_{\edgi_j}^{-1}c_j \labl_{\edgi_{j}})=1$, and realize $g$ in the label of~$\what{\walki}_{\!_{(1)}}=\walki_1\cdots \walki_{j-1}\walki_{j+1}\cdots \walki_s$, a
\bp-\walk\ with shorter elementary decomposition, which will be again alternating after repeating the operation in the above paragraph. Repeating this operations a finite number of times, until no trivial choices are made at the degenerate vertices, we obtain the desired result.
\end{proof}

\begin{rem}
The usefulness of the previous lemma is the following: when realizing an element $g\in \recg{\Ati}$ from the subgroup recognized by a reduced automaton $\Sti$ as $g\in \red{\labl}_{\what{\walki}}$ for some alternating \mbox{\bp-\walk}~$\what{\walki}$, the elementary decomposition of $\what{\walki}$ automatically provides the syllable decomposition of $g$ as element of $\Gpfi*\Gpfii$. This is a crucial bridge between the algebraic and the geometric aspects of the theory.
\end{rem}


One of the most useful applications of reduced $(\Gpfi, \Gpfii)$-automata is that they naturally encode the Kurosh free product decomposition (the induced splitting) of their recognized subgroups as subgroups of $\Gpfi*\Gpfii$. With some technical differences, our exposition follows~\cite{ivanov_intersection_1999,kapovich_foldings_2005}, but with special emphasis on the algorithmic point of view. The theorem below appears as Lemma~4 in Ivanov's~\cite{ivanov_kurosh_2008} and, in a more general setting, as Proposition~4.3 in Kapovich--Weidmann--Myasnikov~\cite{kapovich_foldings_2005}.
A detailed proof (including the algorithmicity in the the finite type case) can be found in~\cite{delgado_rodriguez_extensions_2017}.

\begin{nott}
Let $\Ati$ be a reduced $(\Gpfi, \Gpfii)$-wedge automaton. Fix a maximal subtree $\ATreei$ in $\Ati^{+}$, and let $\Edgi=\Edgi \Ati^{+} \setmin \Edgi \ATreei$ be the set of arcs of $\Ati^{+}$ outside $\ATreei$. For every two vertices $u,v\in V\Ati$, let $\ATreei[u,v]$ denote the unique walk without backtracking from $u$ to $v$ along the tree $\ATreei$. Now, for every vertex $u\in V\Ati$, let $z_u$ denote the basic label $z_u =\labl^{\bullet}_{\scriptscriptstyle{\ATreei}[u,\scriptsize{\bp}]} \in G_1 *G_2$ and, for every arc $\edgi\in \Edgi^{\pm}$, let $x_\edgi$ denote the element $x_\edgi =z_{\start \edgi}^{-1} \labl_{\edgi}\, z_{\term \edgi}\in G_1 *G_2$.
\end{nott}

\begin{thm}\label{thm: from reduced to Kurosh}
Let $\Ati$ be a reduced $(\Gpfi, \Gpfii)$-automaton. Then, with the above notations, the subgroup recognized by $\Ati$ is
 \begin{equation}\label{eq: algorithmic Kurosh decomposition}
\Recg{\Ati} =\FF * \left( \Freeprod_{\vertii\in \Vertexs_1\Ati} z_\vertii^{-1} \, \labl_\vertii \, z_\vertii \right) * \left(\Freeprod_{\vertii\in \Vertexs_2\Ati} z_\vertii^{-1} \, \labl_\vertii \, z_\vertii \right) \,,
 \end{equation}
where $\FF$ is the free subgroup of $\recg{\Ati}$ freely generated by the set $\Set{ x_\edgi \mid \edgi\in \Edgi}$. Moreover, if $\Ati$ is of finite type, then the subgroup $\recg{\Ati}$ is finitely generated, and we can algorithmically compute a Kurosh decomposition like~\eqref{eq: algorithmic Kurosh decomposition} for $\recg{\Ati}\leqslant G_1*G_2$. \qed
\end{thm}

\begin{cor}\label{cor: fg <-> finite type}
For a reduced $(\Gpfi, \Gpfii)$-automaton $\Ati$, the group $\recg{\Ati}$ is finitely generated if and only if $\Ati$ is of finite type (\ie the underlying graph of $\Ati$ has finite rank, and all vertex labels of $\Ati$ are finitely generated). In this case, a set of generators (in the form of the Kurosh decomposition theorem) for $\recg{\Ati}$ is  computable. \qed
\end{cor}


\Cref{thm: from reduced to Kurosh} and~\Cref{cor: fg <-> finite type} are easily seen to be false if we substitute reduced automata by general wedge automata.

\subsection{Effective reduction of wedged automata}

Following~\cite{ivanov_intersection_1999} and \cite{kapovich_foldings_2005}, the next step is to show that every finitely generated subgroup $H\leqslant \Gpfi*\Gpfii$ is the subgroup recognized by some reduced $(\Gpfi, \Gpfii)$-automaton of finite type.
We introduce several elementary operations on wedge automata which will not change their recognized subgroup (to simplify notation, we shall work with positive arcs and assume that everything done to an arc $\edgi$ will also be done accordingly to $\edgi^{-1}$).

\begin{defn} \label{def: wedge transformations}
Let us consider the following elementary transformations on a $(\Gpfi, \Gpfii)$-wedge automata:
\begin{enumerate}[ind]
\item \label{item: adjustement transf}
\defin{Adjustment}: replacing the label of any arc $\edgi$ leaving a secondary vertex $\vertii$, by $c\cdot \labl_{\edgi}$, for any $c\in \labl_{\vertii}$; see~\Cref{fig: adjustment wedge}.

\bigskip
\begin{minipage}[t]{\linewidth}
\centering
\newcount\mycount
\begin{tikzpicture}[shorten >=1pt, node distance=0.5 and 1.75, on grid,auto,>=stealth']
\node[sstate] (q1) {$\scriptstyle{C}$};
\foreach \number in {1,...,2}{\mycount=\number \advance\mycount by -1
\multiply\mycount by 40 \advance\mycount by 0
\node[state] (p\number) at (215+\the\mycount:1.5cm) {};
}
\node[state] (pk) at (325+\the\mycount:1.5cm) {};
\path[->](q1) edge[] node[above left= -0.1] {\scriptsize{$\gpi_1$}}(p1);
\path[->](q1) edge[] node[above left= -0.1] {\scriptsize{$\gpi_2$}}(p2);
\path[->](q1) edge[] node[above right= -0.05] {\scriptsize{$\gpi_k$}}(pk);
\foreach \number in {1,...,3}{\mycount=\number \advance\mycount by -1
\multiply\mycount by 10 \advance\mycount by 0
\node[dot] (d\number) at (280+\the\mycount:0.75cm) {};
}
\node[] (i) [] [below right = of q1] {};
\node[] (f) [] [right = 33pt of i] {};
\path[->,thin,>=to] (i) edge[bend left](f);

\begin{scope}[shift={(4.75,0)}]
\node[sstate] (q'1) {$\scriptstyle{C}$};
\foreach \number in {1,...,2}{\mycount=\number \advance\mycount by -1
\multiply\mycount by 40 \advance\mycount by 0
\node[state] (p'\number) at (215+\the\mycount:1.5cm) {};
}
\node[state] (p'k) at (325+\the\mycount:1.5cm) {};
\path[->](q'1) edge[] node[above left= -0.1] {\scriptsize{$c\gpi_1$}}(p'1);
\path[->](q'1) edge[] node[above left= -0.1] {\scriptsize{$\gpi_2$}} (p'2);
\path[->](q'1) edge[] node[above right= -0.05] {\scriptsize{$\gpi_k$}} (p'k);
\foreach \number in {1,...,3}{\mycount=\number \advance\mycount by -1
\multiply\mycount by 10 \advance\mycount by 0
\node[dot] (d'\number) at (280+\the\mycount:0.75cm) {};
}
\end{scope}
\end{tikzpicture}

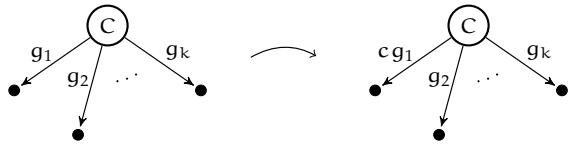
\captionof{figure}{\protect{Adjustment}}\label{fig: adjustment wedge}
\end{minipage}

\item \defin{Conjugation}: (for $\nu=1,2$) replacing, given $\gpi \in G_{\nu}$, the label $\labl_{\vertii}$ of a $\nu$-secondary vertex $\vertii$, by $(\labl_{\vertii})^{\gpi^{-1}}=\gpi \labl_{\vertii} \gpi^{-1}$; and replacing the label $\labl_{\edgi_i}$ of every arc $\edgi_i$ incident from $\vertii$, by the respective $\gpi \labl_{\edgi_i}$; see~\Cref{fig: conjugation wedge}.

\bigskip
\begin{minipage}[t]{\linewidth}
\centering
\newcount\mycount
\begin{tikzpicture}[shorten >=1pt, node distance=0.5 and 1.75, on grid,auto,>=stealth']
\node[sstate] (q1) {$\scriptstyle{C}$};
\foreach \number in {1,...,2}{ \mycount=\number \advance\mycount by -1
\multiply\mycount by 40 \advance\mycount by 0
\node[state] (p\number) at (215+\the\mycount:1.5cm) {};
}
\node[state] (pk) at (325+\the\mycount:1.5cm) {};
\path[->](q1) edge[] node[above left= -0.1] {\scriptsize{$\gpi_1$}}(p1);
\path[->](q1) edge[] node[above left= -0.1] {\scriptsize{$\gpi_2$}}(p2);
\path[->](q1) edge[] node[above right= -0.05] {\scriptsize{$\gpi_k$}}(pk);
\foreach \number in {1,...,3}{\mycount=\number \advance\mycount by -1
\multiply\mycount by 10 \advance\mycount by 0
\node[dot] (d\number) at (280+\the\mycount:0.75cm) {};
}
\node[] (i) [] [below right = of q1] {};
\node[] (f) [] [right = 33pt of i] {};
\path[->,thin,>=to] (i) edge[bend left](f);
\begin{scope}[shift={(4.75,0)}]
\node[sstate] (q'1) {$\scriptstyle{C^{\gpi^{-1} } }$};
\foreach \number in {1,...,2}{ \mycount=\number \advance\mycount by -1
\multiply\mycount by 40 \advance\mycount by 0
\node[state] (p'\number) at (215+\the\mycount:1.65cm) {};
}
\node[state] (p'k) at (325+\the\mycount:1.5cm) {};
\path[->](q'1) edge[] node[above left= -0.1] {\scriptsize{$\gpi \gpi_1$}}(p'1);
\path[->](q'1) edge[] node[above left= -0.1] {\scriptsize{$\gpi \gpi_2$}}(p'2);
\path[->](q'1) edge[] node[above right= -0.05] {\scriptsize{$\gpi \gpi_k$}}(p'k);
\foreach \number in {1,...,3}{ \mycount=\number
\advance\mycount by -1 \multiply\mycount by 10 \advance\mycount by 0
\node[dot] (d'\number) at (280+\the\mycount:0.85cm) {};
}
\end{scope}
\end{tikzpicture}
\captionof{figure}{\protect{Conjugation}}\label{fig: conjugation wedge}
\medskip
\end{minipage}

\item \label{item: wedge automata isolation}\defin{Isolation}: removing from $\Ati$ all the connected components not containing the basepoint; see~\Cref{fig: isolation wedge}.

\medskip
\begin{minipage}[t]{\linewidth}
\centering
\begin{tikzpicture}
\node[state,accepting] (0) at (0,0) {};
\coordinate (1) at (1,0.5);
\coordinate (2) at (1,-0.3);
\coordinate (3) at (-1,0.3);
\coordinate (4) at (-0.9,-0.4);
\node[state,accepting] (a) at (4.3,0) {};
\newsavebox{\basepoint}
\savebox{\basepoint}{\draw[rounded corners=0.5mm] (0) \irregularcircle{0.5cm}{2mm};}
\draw (0,0) node {\usebox{\basepoint}};
\draw[rounded corners=0.2mm] (1) \irregularcircle{0.3cm}{1mm};
\draw[rounded corners=0.2mm] (2) \irregularcircle{0.2cm}{0.5mm};
\draw[rounded corners=0.2mm] (3) \irregularcircle{0.2cm}{0.5mm};
\draw[rounded corners=0.2mm] (4) \irregularcircle{0.1cm}{0.3mm};
\node[] (i) [] at (1.8,0) {};
\node[] (f) [] [right = 33pt of i] {};
\path[->,thin,>=to](i) edge[bend left](f);
\draw (4.3,0) node {\usebox{\basepoint}};
\node[] (c) [right = 1.5 of f] {};
\end{tikzpicture}
\captionof{figure}{\protect{Isolation}}\label{fig: isolation wedge}
\medskip
\end{minipage}

\item \label{item: primary open folding} \defin{Primary open folding}: for $\nu=1,2$, given two $\nu$-secondary vertices $\vertii_1,\vertii_2$ adjacent to the same primary vertex through respective arcs $\edgi_1,\edgi_2$ with the same label $\gpi \in G_{\nu}$, identify $\vertii_1$ and $\vertii_2$ into a new secondary vertex with label $ \gen{\labl_{\vertii_1}, \labl_{\vertii_2}}$, and identify the arcs $\edgi_1,\edgi_2$ into a new arc with the same label $\gpi$; see~\Cref{fig: primary open folding}.

\smallskip
\begin{minipage}[t]{\linewidth}
\centering
\begin{tikzpicture}[shorten >=1pt, node distance=0.5 and 1.5, on grid,auto,>=stealth']
  \node[state] (q1) {};

   \node[sstate] (p0) [above left = 1.3 and 0.7 of q1] {$\scriptstyle{C_1}$};
   \node[sstate] (p00) [above right = 1.3 and 0.7 of q1] {$\scriptstyle{C_2}$};

  \foreach \number [count=\count from 0] in {1,...,2}{
       \node[] (p\number) at (215+\count*40:0.7cm) {};
    }

    \node[] (pk) at (325:0.7cm) {};

    \foreach \n [count=\count from 0] in {1,...,2}{
       \node[] (p0\n) at ($(p0)+(175-\count*40:0.9cm)$) {};
        \path[->,red!50] (p0) edge[] (p0\n);
    }
    \node[] (p0k) at ($(p0)+(70:0.9cm)$) {};
    \path[->,red!50] (p0) edge[] (p0k);
    \foreach \n [count=\count from 0] in {1,...,3}{
       \node[dot,red!50] (pp0\n) at ($(p0)+(110-\count*10:0.5cm)$) {};
    }

    \foreach \n [count=\count from 0] in {1,...,2}{
       \node[] (p0\n) at ($(p00)+(5+\count*40:0.9cm)$) {};
        \path[->,blue!50] (p00) edge[] (p0\n);
    }
    \node[] (p00k) at ($(p00)+(110:0.9cm)$) {};
    \path[->,blue!50] (p00) edge[] (p00k);
    \foreach \n [count=\count from 0] in {1,...,3}{
       \node[dot,blue!50] (pp0\n) at ($(p00)+(90-\count*10:0.5cm)$) {};
    }

    \path[->]
        (p0) edge[]
            node[below left = -0.1] {$\gpi$}
            (q1);

        \path[->]
        (p00) edge[]
            node[below right = -0.07] {$\gpi$}
            (q1);

    \path[->]
        (p1) edge[]
            (q1);

    \path[->]
        (p2) edge[]
            (q1);

    \path[->]
        (pk) edge[]
            (q1);

    \foreach \number in {1,...,3}{
        \mycount=\number
        \advance\mycount by -1
  \multiply\mycount by 10
        \advance\mycount by 0
      \node[dot] (d\number) at (280+\the\mycount:0.5cm) {};
    }

   \node[] (i) [] [right = 1.6 of q1] {};
   \node[] (f) [] [right = 33pt of i] {};

   \path[->,thin,>=to]
        (i) edge[bend left]
         (f);

    \node[state] (q'1) [right = of f] {};

    \node[sstate] (p'0) [above = 1.3 of q'1] {$\scriptstyle{C}$};

    \path[->]
        (p'0) edge[]
            node[left] {$\gpi$}
            (q'1);

  \foreach \number [count=\count from 0] in {1,...,2}{
       \node[] (p'\number) at ($(q'1)+(215+\count*40:0.7cm)$) {};
    }

    \node[] (p'k) at ($(q'1)+(325:0.7cm)$) {};


    \path[->]
        (p'1) edge[]
            (q'1);

    \path[->]
        (p'2) edge[]
            (q'1);

    \path[->]
        (p'k) edge[]
            (q'1);

    \foreach \number in {1,...,3}{
        \mycount=\number
        \advance\mycount by -1
  \multiply\mycount by 10
        \advance\mycount by 0
      \node[dot] (d\number) at ($(q'1)+(280+\the\mycount:0.5cm)$) {};
    }

\foreach \n [count=\count from 0] in {1,...,2}{
       \node[] (p'0\n) at ($(p'0)+(160-\count*40:0.9cm)$) {};
        \path[->,red!50] (p'0) edge[] (p'0\n);
    }
    \node[] (p'0k) at ($(p'0)+(-180:0.9cm)$) {};
    \path[->,red!50] (p'0) edge[] (p'0k);
    \foreach \n [count=\count from 0] in {1,...,3}{
       \node[dot,red!50] (pp'0\n) at ($(p'0)+(145-\count*7:0.5cm)$) {};
    }

    \node[] (p'00k) at ($(p'0)+(0:0.9cm)$) {};
    \path[->,blue!50] (p'0) edge[] (p'00k);
    \foreach \n [count=\count from 0] in {1,...,2}{
       \node[] (p'0\n) at ($(p'0)+(20+\count*40:0.9cm)$) {};
        \path[->,blue!50] (p'0) edge[] (p'0\n);
    }

    \foreach \n [count=\count from 0] in {1,...,3}{
       \node[dot,blue!50] (pp'00\n) at ($(p'0)+(33+\count*7:0.5cm)$) {};
    }
\end{tikzpicture}
 
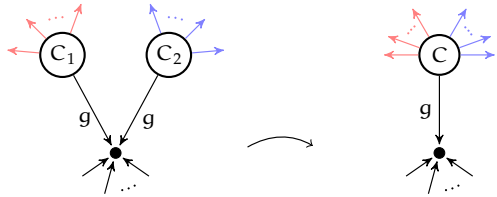
\captionof{figure}{\protect{Primary open folding, where $C = \gen{C_1 \cup C_2}$}}
 \label{fig: primary open folding}
 \medskip
 \end{minipage}

\item \label{item: secondary open folding}
\defin{Secondary open folding:} for $\nu=1,2$, given a $\nu$-secondary vertex $\vertii$ adjacent to two different primary vertices $\verti_1,\verti_2$ through arcs $\edgi_1,\edgi_2$ having the same label $\gpi \in G_{\nu}$, identify the vertices $\verti_1$ and $\verti_2$, and the arcs $\edgi_1,\edgi_2$ into an arc maintaining the label $\gpi$; see~\Cref{fig: secondary open folding}.

\smallskip
\begin{minipage}[t]{\linewidth}
\centering
\begin{tikzpicture}[shorten >=1pt, node distance=0.5 and 1.5, on grid,auto,>=stealth']
\node[sstate] (q0) {$\scriptstyle{C}$};
\node[state] (p1) [below left = 1.3 and 0.7 of q1] {};
\node[state] (p2) [below right = 1.3 and 0.7 of q1] {};
\newcommand\N{2}
\newcommand{\eee}{90}
\newcommand{\bbb}{30}
\foreach \n  in {1,...,\N,4}{
\node[] (q\n) at ($(q0)+(\eee - \bbb*5/2+\n*\bbb:0.9cm)$) {};
\path[->] (q0) edge[] (q\n);
}
\foreach \m  in {1,...,3}{
\node[dot] (q0\m) at ($(q0)+(\eee- \bbb/5 +\m*\bbb/3:0.5cm)$) {};
}
\path[->](q0) edge[] node[left] {$\gpi$} (p1);
\path[->](q0) edge[] node[right] {$\gpi$} (p2);
\path[-,dashed, gray](p1) edge[out=50 , in=130 ,min distance=12mm] node[below] {$\scriptstyle{\trivial}$}(p2);
\renewcommand\N{2}
\renewcommand{\eee}{255}
\renewcommand{\bbb}{25}
\foreach \n  in {1,...,\N,4}{
\node[] (p1\n) at ($(p1)+(\eee - \bbb*5/2+\n*\bbb:0.8cm)$) {};
\path[->,red!50] (p1) edge[] (p1\n);
}
\foreach \m  in {1,...,3}{
\node[dot,red!50] (p1\m) at ($(p1)+(\eee- \bbb/5 +\m*\bbb/3:0.5cm)$) {};
}
\renewcommand\N{2}
\renewcommand{\eee}{280}
\renewcommand{\bbb}{25}
\foreach \n  in {1,...,\N,4}{
\node[] (p2\n) at ($(p2)+(\eee - \bbb*5/2+\n*\bbb:0.8cm)$) {};
\path[->,blue!50] (p2) edge[] (p2\n);
}
\foreach \m  in {1,...,3}{
\node[dot,blue!50] (p2\m) at ($(p2)+(\eee- \bbb/5 +\m*\bbb/3:0.5cm)$) {};
}
\node[] (i) [] [below right = of q0] {};
\node[] (f) [] [right = 33pt of i] {};
\path[->,thin,>=to] (i) edge[bend left] (f);
\node[sstate] (q'0) [above right = of f] {$\scriptstyle{C}$};
\node[state] (p'1) [below = 1.3 of q'0] {};
\path[->](q'0) edge[] node[left] {$\gpi$} (p'1);
\renewcommand\N{2}
\renewcommand{\eee}{90}
\renewcommand{\bbb}{30}
\foreach \n  in {1,...,\N,4}{
\node[] (q'\n) at ($(q'0)+(\eee - \bbb*5/2+\n*\bbb:0.9cm)$) {};
\path[->] (q'0) edge[] (q'\n);
}
\foreach \m  in {1,...,3}{
\node[dot] (q'0\m) at ($(q'0)+(\eee- \bbb/5 +\m*\bbb/3:0.5cm)$) {};
}
\renewcommand\N{2}
\renewcommand{\eee}{225}
\renewcommand{\bbb}{20}
\foreach \n  in {1,...,\N,4}{
\node[] (p'1\n) at ($(p'1)+(\eee - \bbb*5/2+\n*\bbb:0.8cm)$) {};
\path[->,red!50] (p'1) edge[] (p'1\n);
}
\foreach \m  in {1,...,3}{
\node[dot,red!50] (p'1\m) at ($(p'1)+(\eee- \bbb/5 +\m*\bbb/3:0.5cm)$) {};
}
\renewcommand\N{2}
\renewcommand{\eee}{310}
\renewcommand{\bbb}{20}
\foreach \n  in {1,...,\N,4}{
\node[] (p'1\n) at ($(p'1)+(\eee - \bbb*5/2+\n*\bbb:0.8cm)$) {};
\path[->,blue!50] (p'1) edge[] (p'1\n);
}
\foreach \m  in {1,...,3}{
\node[dot,blue!50] (p'1\m) at ($(p'1)+(\eee- \bbb/5 +\m*\bbb/3:0.5cm)$) {};
}
\end{tikzpicture}

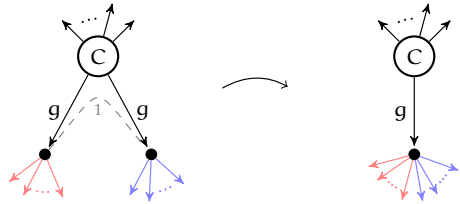
\captionof{figure}{\protect{{Secondary open folding}}} \label{fig: secondary open folding}
\medskip
\end{minipage}

\item \label{item: wedge automata closed folding} \defin{Closed folding}: for $\nu=1,2$, given a primary vertex $\verti$ adjacent to a $\nu$-secondary vertex $\vertii$ by two non mutually inverse arcs $\edgi_1,\edgi_2$, consists in identifying $\edgi_1$ and $\edgi_2$ into a single arc with label $\labl_{\edgi_1}$, and change the label of $\vertii$ from $\labl_\vertii$ to $\langle \labl_\vertii,\, \labl_{\edgi_1 }\labl_{\edgi_2}^{-1}\rangle\leqslant G_{\nu}$; see~\Cref{fig: closed folding}.

\smallskip
\begin{minipage}[t]{\linewidth}
\centering
\begin{tikzpicture}[shorten >=1pt, node distance=0.5 and 1.5, on grid,auto,>=stealth']
  \node[sstate] (q0) {$\scriptstyle{C}$};

   \node[state] (p1) [below = 1.3 of q0] {};

\newcommand\N{2}
\newcommand{\eee}{90}
\newcommand{\bbb}{30}

\foreach \n  in {1,...,\N,4}{
       \node[] (q\n) at ($(q0)+(\eee - \bbb*5/2+\n*\bbb:0.9cm)$) {};
        \path[->] (q0) edge[] (q\n);
    }
\foreach \m  in {1,...,3}{
       \node[dot] (q0\m) at ($(q0)+(\eee- \bbb/5 +\m*\bbb/3:0.5cm)$) {};
    }

\path[->]
        (q0) edge[bend right]
            node[left] {${\gpi}$}
            (p1);

\path[->]
        (q0) edge[bend left]
            node[right] {${\gpii}$}
            (p1);

\renewcommand\N{2}
\renewcommand{\eee}{270}
\renewcommand{\bbb}{25}

\foreach \n  in {1,...,\N,4}{
       \node[] (p1\n) at ($(p1)+(\eee - \bbb*5/2+\n*\bbb:0.8cm)$) {};
        \path[->,red!50] (p1) edge[] (p1\n);
    }
\foreach \m  in {1,...,3}{
       \node[dot,red!50] (p1\m) at ($(p1)+(\eee- \bbb/5 +\m*\bbb/3:0.5cm)$) {};
    }

   \node[] (i) [] [below right = of q0] {};
   \node[] (f) [] [right = 33pt of i] {};

   \path[->,thin,>=to]
        (i) edge[bend left]
         (f);

    \node[sstate] (q'0) [above right = of f] {$\scriptstyle{C'}$};

        \node[state] (p'1) [below = 1.3 of q'0] {};

    \path[->]
        (q'0) edge[]
            node[left] {${\gpi}$}
            (p'1);

\renewcommand\N{2}
\renewcommand{\eee}{90}
\renewcommand{\bbb}{30}

\foreach \n  in {1,...,\N,4}{
       \node[] (q'\n) at ($(q'0)+(\eee - \bbb*5/2+\n*\bbb:0.9cm)$) {};
        \path[->] (q'0) edge[] (q'\n);
    }

\foreach \m  in {1,...,3}{
       \node[dot] (q'0\m) at ($(q'0)+(\eee- \bbb/5 +\m*\bbb/3:0.5cm)$) {};
    }

\renewcommand\N{2}
\renewcommand{\eee}{270}
\renewcommand{\bbb}{25}

\foreach \n  in {1,...,\N,4}{
       \node[] (p'1\n) at ($(p'1)+(\eee - \bbb*5/2+\n*\bbb:0.8cm)$) {};
        \path[->,red!50] (p'1) edge[] (p'1\n);
    }

\foreach \m  in {1,...,3}{
       \node[dot,red!50] (p'1\m) at ($(p'1)+(\eee- \bbb/5 +\m*\bbb/3:0.5cm)$) {};
    }

\end{tikzpicture}
 \captionof{figure}{\protect{Closed folding, where $C'=\gen{C,\gpi\gpii^{-1}}$}}
 \label{fig: closed folding}
 \medskip
 \end{minipage}
\end{enumerate}
\end{defn}

Note that the (folding) transformations \ref{item: primary open folding}, \ref{item: secondary open folding}, and \ref{item: wedge automata closed folding} in \Cref{def: wedge transformations} decrease the number of arcs in the automata exactly by $1$.
The following result is also straightforward.

\begin{lem}\label{lem: preserving transformations}
Let $\Ati$ be a reduced $(\Gpfi, \Gpfii)$-automaton, and let $\Ati\! \xtransf{\!\!} \Ati'$ be any of the elementary transformations in~\Cref{def: wedge transformations}. Then, the recognized subgroups of $\Ati$ and~$\Ati'$ coincide. \qed
\end{lem}


\begin{thm}[\longcite{ivanov_intersection_1999}]\label{thm: wedge reduced construction}
For any groups $\Gpfi, \Gpfii$, and any finitely generated subgroup $\Sgpi \leqslant \Gpfi * \Gpfii$, there exists a reduced $(\Gpfi, \Gpfii)$-automaton recognizing $\Sgpi$. Moreover, if both $\Gpfi$ and $\Gpfii$ have solvable membership problem, then given a finite set of generators of a  subgroup~$H\leqslant \Gpfi*\Gpfii$, one can algorithmically obtain a reduced $(\Gpfi, \Gpfii)$-automaton of finite type recognizing~$\Sgpi$.
\end{thm}

\begin{proof}
Since the trivial automaton recognizes the trivial subgroup, we can assume $H\neq 1$.
Starting with the corresponding flower automaton for $H$, successively do the following:
\begin{enumerate}[(I)]
\item if, for $\nu=1,2$, $\Ati$ has a primary vertex adjacent to two different $\nu$-secondary vertices, then apply a suitable conjugation move to one of the $\nu$- vertices and then a primary open folding.
\item if, for $\nu=1,2$, $\Ati$ has a primary vertex and a $\nu$-secondary vertex connected to each other by two non mutually inverse arcs then apply a closed folding.
\item if $\Ati$ has a secondary vertex $\vertii$ adjacent to two different primary vertices through arcs $\edgi_1$ and $\edgi_2$, $\start \edgi_1 =\start \edgi_2 =\vertii$, such that $\labl_{\edgi_1} \labl_{\edgi_2}^{-1}\in \labl_\vertii$, apply a suitable adjustment move to one of the two arcs, and then an elementary open folding.
\end{enumerate}

Note that this can be done algorithmically (we use MP in $G_1$ and $G_2$ for (III)). Since the number of edges decreases at each step, this process will eventually stop. It is clear that the resulting automaton is reduced, of finite type, and it recognizes $H$ by \Cref{lem: preserving transformations}.
\end{proof}

\subsection{A reduced automaton for the intersection}

Recall that if $\Gpfi$ and $\Gpfii$ are Howson, then $\Gpfi*\Gpfii$ is Howson as well; see~\cite{baumslag_intersections_1966} and~\cite{ivanov_intersection_1999}, the second proof being essentially the one we present here.

The goal of this section is to describe, following~\cite{ivanov_intersection_1999}, a reduced automaton ${\junction{H}{K}}$ recognizing the intersection $ H\cap K$, in terms of given reduced automata $\Ati_{\!H}$ and $\Ati_{\!K}$, with $\recg{\Ati_{\!H}}=H$ and $\recg{\Ati_{\!K}}=K$, where $H,K\leqslant G_1 *G_2$ are finitely generated subgroups. This construction is not algorithmic, in general, since $\junction{H}{K}$ may very well be not of finite type, even when both $\Ati_{\!H} $ and $\Ati_{\!K} $ are so (corresponding to the case where $H,K$ are both finitely generated but $H\cap K$ is not, see~\Cref{cor: fg <-> finite type}).

Later, in Section~\ref{sec: free product theorems}, we shall give an effective procedure which starts constructing, locally, the aforementioned automaton $\junction{H}{K}$. While running, there will be an alert observing the construction: if some (algorithmically checkable) specific situation occurs, then the intersection $H\cap K$ is not finitely generated. We shall achieve our goal by proving that, in finite time, either the alert sounds or the procedure terminates providing the finite type reduced automaton~$\junction{H}{K}$ as output.

\bigskip

So, suppose we are given two finite type reduced $(\Gpfi, \Gpfii)$-wedge automata $\Ati_{\!H}$ and $\Ati_{\!K}$ with recognized subgroups $\recg{\Ati_{\!H}}=H$ and $\recg{\Ati_{\!k}}=K$, where $H,K\leqslant \Gpfi*\Gpfii$ are finitely generated subgroups. Along the following paragraphs, we will first define the \emph{product automaton} $\Atipr$ of $\Ati_{\!H}$ and $\Ati_{\!K}$ whose main connected component will be the junction $\junction{H}{K}$.

Define the set of \defin{primary vertices of the product} $\Atipr$ as the cartesian product
$
\Vertexs_{0}\Atipr =\Vertexs_{0}\Ati_{\!H} \times \Vertexs_{0}\Ati_{\!K} .
$ 
and the basepoint \bp\ of $\Atipr$ to be the pair of basepoints, \ie $\bp = (\bp_H, \bp_K)$.
Now, for $\nu=1,2$, we consider the subsets of primary vertices (in $\Ati_{\!H} $ and $\Ati_{\!K} $) adjacent to some $\nu$-secondary vertex,
\ie $
\Vertexs_{_{{0}\leftarrow \nu}} \, \Ati_{\!H}
=
\set{\, \verti\in \Vertexs_{0}\,\Ati_{\!H} \mid \verti \text{ is adj. to a $\nu$-secondary in } \Ati_{H}\,}
\subseteq
\Vertexs_{0}\,\Ati_{\!H}\,
$
(and idem for $\Vertexs_{_{{0}\leftarrow \nu}} \, \Ati_{\!K}$)
and define the relation $\equiv_{\nu}$ on the set $\Vertexs_{_{{0}\leftarrow \nu}}\,\Ati_{\!H} \times \Vertexs_{_{{0}\leftarrow \nu}}\,\Ati_{\!K}  \subseteq \Vertexs_{0}\Atipr$ to be: $(\verti_1, \verti'_1)\equiv_{\nu} (\verti_2, \verti'_2)$ if and only if there exist two $\nu$-elementary walks: $\walki$  from $\verti_1$ to $\verti_2$ in $\Ati_{\!H}$ (say $\walki =\verti_1\edgi_1^{-1} \vertii \edgi_2\verti_2$, with $\vertii\in \Vertexs_{\nu}\,\Ati_{\!H} $), and $\walki'$ from $\verti'_1$ to $\verti'_2$ in $\Ati_{\!K}$ (say~$\walki' =\verti'_1\edgi^{\prime -1}_1\vertii'\edgi'_2\verti'_2$, with $\vertii'\in \Vertexs_{\nu}\,\Ati_{\!K} $), such that the intersection of their labels is nonempty, $\labl_{\gamma }\cap \labl_{\gamma'}\neq \varnothing$.

The following result is not hard to see with the natural reasoning.
\begin{lem}\label{lem: nu-equivalence secondary vertices}
The relations $\equiv_1$ and $\equiv_2$ are equivalence relations. \qed
\end{lem}

Now, define the $\nu$-secondary vertices of $\Atipr$ to be the equivalence classes modulo $\equiv_{\nu}$, namely
$
\Vertexs_{\nu}\, \Atipr =
\allowbreak
(\Vertexs_{_{{0}\leftarrow \nu}}\, \Ati_{\!H} \times \Vertexs_{_{{0}\leftarrow \nu}}\, \Ati_{\!K} )\, /  \equiv_{\nu}
$. 
Finally, define the $\nu$-arcs of $\Atipr$ as
$
\Edgi_{\nu} \Atipr =  \set{\,\bsf{q} \xarc{} \bsf{p} \mid \bsf{q}\in \Vertexs_{\nu}\,\Atipr,\ \bsf{p}\in \bsf{q}\,}
$, 
\ie for each secondary vertex $\bsf{q} \in \Vertexs_{\nu} \Atipr$, and each primary vertex $\bsf{p}=(\verti,\verti')\in \bsf{q}$, add a $\nu$-arc $\bsf{q} \xarc{} \bsf{p}$. This finishes the definition of the underlying digraph of $\Atipr$. Observe that $\Atipr$ may not be connected in general, even with $\Ati_{\!H} $ and $\Ati_{\!K} $ being so.

By construction, it is clear that, for $\nu =1,2$, every primary vertex of $\Atipr$ is adjacent to at most one $\nu$-secondary vertex of $\Atipr$ through at most one arc. We define now natural projections (digraph homomorphisms) from $\Atipr$ to $\Ati_{\!H}$ and $\Ati_{\!K}$. Define $\pi \colon \Atipr \to \Ati_{\!H}$ as follows: for primary vertices, take the projection to the first coordinate, $\pi \colon \Vertexs_{0}\Atipr \to \Vertexs_{0}\Ati_{\!H}$, $(\verti,\verti') \mapsto \verti$; for $\nu$-secondary vertices, assign to every vertex $\bsf{q} \in \Vertexs_{\nu}\,\Atipr$ the only $\nu$-secondary vertex in $\Ati_{\!H} $ adjacent to every $\verti = (\verti,\verti')\pi \in \Vertexs_{0}\,\Ati_{\!H}$ for $(\verti,\verti')\in \bsf{q}$; finally, for $\nu$-arcs,
assign to every $\bsf{q} \xarc{} \bsf{p}$ (for $\bsf{q} \in \Vertexs_{\nu}\Atipr$ and $(\verti,\verti')\in \bsf{q}$) the unique $\nu$-arc in $\Ati_{\!H}$ from $\vertii= \bsf{q}\pi$ to $\verti= (\verti,\verti')\pi$.
Clearly, $\pi$ is a well defined digraph homomorphism, called the \defin{projection} to $\Ati_{\!H} $. The projection to $\Ati_{\!K} $, denoted by $\pi'\colon \Atipr\to \Ati_{\!K} $, is defined in an analogous way.

It remains to establish the labels for its vertices and arcs. For $\nu=1,2$, and for every $\nu$-secondary vertex $\bsf{q}\in \Vertexs_{\nu}\Atipr$, choose a distinguished primary vertex $\bsf{p}_{\bsf{q}} = (\verti_{\bsf{q}},\verti'_{\bsf{q}})\in \bsf{q}$, and let $\bsf{e}_{\bsf{q}}$ be the only arc in $\Atipr$ form $\bsf{q}$ to the representative $\bsf{p}_{\bsf{q}}$. This means that, in $\Ati_{\!H}$, there is a $\nu$-arc $\edgi_{\bsf{q}} \coloneqq (\bsf{e}_{\bsf{q}}) \pi \in \Edgi_{\nu}\,\Ati_{\!H}$ from $\vertii=\bsf{q} \pi \in \Vertexs_{\nu}\,\Ati_{\!H}$ to $\verti_{\bsf{q}}$, and, in $\Ati_{\!K}$, there is a $\nu$-arc $\edgi'_{\bsf{q}} \coloneqq (\bsf{e}_{\bsf{q}}) \pi'  \in \Edgi_{\nu}\Ati_{\!K}$ from $\vertii'=\bsf{q} \pi' \in\Vertexs_{\nu}\Ati_{\!K}$ to $\verti'_{\bsf{q}}$. Then, we define the label of vertex $\bsf{q}$ as
 \begin{equation}\label{label-c}
\labl_{\bsf{q}} \,:=\, \labl_{\edgi_{\bsf{q}}}^{-1}\, \labl_\vertii \, \labl_{\edgi_{\bsf{q}}} \,\cap\, \labl_{\edgi'_{\bsf{q}}}^{-1}\, \labl_{\vertii'} \, \labl_{\edgi'_{\bsf{q}}} = \labl_\vertii^{\,\labl_{\edgi_{\bsf{q}}}} \,\cap\,
\labl_{\vertii'}^{\,\labl_{\edgi'_{\bsf{q}}}} \,\leqslant\, G_{\nu} \,.
 \end{equation}
Finally, for any $\nu$-arc $\bsf{e} \in \Edgi_{\nu}\,\Atipr$ from $\bsf{q}$ to a primary vertex $\bsf{p} = (\verti,\verti') \in \bsf{q}$, call $\edgi \coloneqq \bsf{e}\pi$, $\edgi' \coloneqq \bsf{e}\pi'$, and define the label of $\bsf{e}$ as an arbitrary element from the coset intersection $\labl_{\bsf{e}} \in \labl_{ \edgi_{\bsf{q}}}^{-1} \labl_\vertii \labl_{\edgi} \cap \labl_{ \edgi'_{\bsf{q}}}^{-1} \labl_{\vertii'} \labl_{\edgi'}$, which is nonempty since $(\verti_{\bsf{q}}, \verti'_{\bsf{q}})\equiv_{\nu} (\verti,\verti')$ by construction; see~\Cref{fig: secondary vertex and labels in Atipr}. Note that, in particular, we can take $\labl_{\bsf{e}_{\bsf{q}}}=1$.

\ffigure{
\begin{tikzpicture}[shorten >=1pt, node distance=0.75 and 0.75, on grid,auto,>=stealth',baseline=0pt]
\node[] (0) [] {};
\node[state] (p1) [below = 1 of 0]{};
\node[sstate1] (q1) [below left = of p1] {$\scriptstyle{C}$};
\node[state] (p2) [below right = of q1] {};
\node[state] (p'1) [right = 1 of 0] {};
\node[sstate1] (q'1) [above right =  of p'1] {$\scriptstyle{C'}$};
\node[state] (p'2) [below right =   of q'1] {};
\node[] (aux) [right = 1 of p1] {};
\node[state] (pp1) [below = 1 of p'2] {};
\node[state] (pp2) [right= 1 of p2] {};
\node[sstate1] (qq) [below right= 3.5 and 3.5 of 0] {$\scriptstyle{C^{g_1} \cap C'^{g'_1}}$};
\path[->] (q1) edge[] node[pos=0.4,above left=-0.05] {${\gpi_2}$} (p1);
\path[->] (q1) edge[] node[pos=0.4,below left=-0.05] {${\gpi_1}$} (p2);
\path[->] (q'1) edge[] node[pos=0.4,above left=-0.05] {${\gpi'_1}$} (p'1);
\path[->] (q'1) edge[] node[pos=0.4,above right=-0.05] {${\gpi'_2}$} (p'2);
\path[->] (qq) edge[] node[pos=0.4,above] {${\trivial}$} (pp2);
\path[->] (qq) edge[] node[pos=0.4,above right] {${g \ \in\  g_1^{-1} C g_2 \ \cap\ {g'_1}^{-1} C' g'_2}$} (pp1);
\node[draw,circle,rounded corners,thick,dashed,red!50,fit=(pp1) (pp2)] {};
\end{tikzpicture}
}
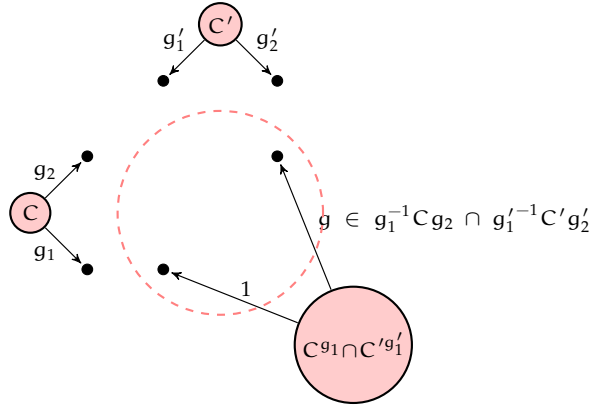
{Secondary vertex and related labels in $\Atipr$} {fig: secondary vertex and labels in Atipr}

This completes the definition of the product automaton $\Atipr$ (technically depending on some choices made on the way). The main property of the labels defined for $\Atipr$ is expressed in the following lemma.

\begin{lem}\label{lem: wedge label}
Let $\Atipr$ be the product of two reduced wedge automata $\Ati_{\!H}$ and $\Ati_{\!K}$. Then, for every $\nu$-elementary \walk\ $\bm{\walki}$ in $\Atipr$, the projected \walk s $\bm{\walki}\pi$ and $\bm{\walki}\pi'$ are \mbox{$\nu$-elementary} in $\Ati_{\!H}$ and $\Ati_{\!K}$, respectively, and we have $\labl_{\bm{\walki}}=\labl_{\bm{\walki}\pi} \cap \labl_{\bm{\walki}\pi'}$. Furthermore, $\bm{\walki}$ is degenerate if and only if both $\bm{\walki}\pi$ and $\bm{\walki}\pi'$ are degenerate.
\end{lem}

\begin{proof}
Clearly, the projections by $\pi$ and $\pi'$ of $\nu$-elementary \walk s in $\Atipr$ are also \mbox{$\nu$-elementary} \walk s (in $\Ati_{\!H} $ and $\Ati_{\!K} $, respectively), with the original one being degenerate if and only if both projections are degenerate (note that $\bm{\walki}$ may be nondegenerate with one (and only one) of $\bm{\walki}\pi$ and $\bm{\walki}\pi'$ being degenerate; see~\Cref{fig: secondary vertex and labels in Atipr}).

To see the equality in labels, let $\bm{\walki}=\bsf{\verti_1} \bsf{\edgi}^{-1}_{\bsf{1}}\bsf{q}\bsf{\edgi_2} \bsf{\verti_2}$ be a $\nu$-elemen\-ta\-ry \walk\ in $\Atipr$, where $\bsf{\verti_1} =(\verti_1,\verti'_1)$ and $\bsf{\verti_2}=(\verti_2,\verti'_2)$, and let $\bm{\walki}\pi=\verti_1  \edgi_1^{-1} \vertii \edgi_2 \verti_2$ and $\bm{\walki}\pi'=\verti'_1  {(\edgi'_1)}^{-1}\vertii'\edgi'_2 \verti'_2$ be the corresponding \mbox{$\nu$-elementary} \walk s in $\Ati_{\!H} $ and $\Ati_{\!K} $, respectively. We have $\labl_{\bm{\walki}\pi}= \labl_{\edgi_1}^{-1} \,\labl_\vertii\, \labl_{\edgi_2}$ and $\labl_{\bm{\walki}\pi'} = \labl_{ \edgi'_1}^{-1}\, \labl_{\vertii'}\, \labl_{\edgi'_2}$.

Now, consider the distinguished $\nu$-arc $\bsf{e}_{\bsf{q}}$ incident to $\bsf{q}$ (possibly equal to $\bsf{e_1}$ and/or $\bsf{e_2}$). According to the above definitions, we have $\labl_{\bsf{q}}=\labl_\vertii^{\labl_{\edgi_{\bsf{q}}}} \cap \labl_{\vertii'}^{\labl_{\edgi_{\bsf{q'}}}}$ and $\labl_{\bsf{\edgi_i}} \in \labl_{\edgi_{\bsf{q}}}^{-1} \labl_\vertii \labl_{\edgi_i} \cap \labl_{\edgi_{\bsf{q}'}}^{-1} \labl_{\vertii'} \labl_{ \edgi'_i}$, for $i=1,2$; see~\Cref{fig: secondary vertex and labels in Atipr}. Therefore,
 \begin{equation}
\labl_{\bm{\walki}} =\labl_{\bsf{\edgi_1}}^{-1} \labl_{\bsf{q}} \labl_{\bsf{\edgi_2}} =\labl_{\bsf{\edgi_1}}^{-1}\cdot \left(\labl_\vertii^{\labl_{\edgi_{\bsf{q}}}} \cap \labl_{\vertii'}^{\labl_{\edgi_{\bsf{q}'}}}\right)\cdot \labl_{\bsf{\edgi_2}}
= \labl_{\edgi_1}^{-1} \labl_\vertii \labl_{\edgi_2} \cap \labl_{ \edgi'_1}^{-1} \labl_{\vertii'} \labl_{\edgi'_2}=\labl_{\bm{\walki}\pi}\cap \labl_{\bm{\walki}\pi'}.
 \end{equation}
This completes the proof.
\end{proof}

We can now state the definition of junction automaton $\junction{H}{K}$ and show that it is a reduced $(\Gpfi,\Gpfii)$-wedge automaton such that~$\recg{\junction{H}{K}} =H\cap K$.

\begin{defn}
With the above notation, we define the \defin{junction automaton} of $\Ati_{\!H}$ and $\Ati_{\!K}$, denoted by $\junction{H}{K}$, as the connected component of $\Atipr$ containing the basepoint. (Recall that there are some arbitrary choices made on the way so, $\junction{H}{K}$ is not canonically associated to $\Ati_{\!H}$ and $\Ati_{\!K}$.)
\end{defn}

\begin{prop}\label{prop: junction recognizes itersection}
The junction automaton $\junction{H}{K}$ is a $(\Gpfi,\Gpfii)$-reduced automaton recognizing the subgroup $H\cap K$.
\end{prop}

\begin{proof}
It is clear that the junction automaton satisfies the properties~(i) and~(ii) in~\Cref{def: reduced wedge automata}.

To see property~(iii), take $\nu=1,2$, let $\bsf{q}\in \Vertexs_{\nu}\,\Atipr$ be a $\nu$-secondary vertex of $\Atipr$, let $\bsf{e_1}, \bsf{e_2}$ be two \emph{different} $\nu$-arcs from $\bsf{q}$ to $\bsf{\verti_1}=(\verti_1,\verti_1')$ and $\bsf{\verti_2}= (\verti_2,\verti_2')$, respectively, and consider the nondegenerate elementary \walk\ $\bm{\walki} =\bsf{p}\bsf{e}^{-1}_{\bsf{1}}\bsf{q}\bsf{e_2} \bsf{\verti_2}$. By symmetry, we can assume $\verti_1 \neq \verti_2$, \ie that $\bm{\walki}\pi =\verti_1 \edgi_1^{-1} \vertii  \edgi_2\verti_2$ is a nondegenerate elementary \walk\ in $\Ati_{\!H} $. Since $\Ati_{\!H} $ is reduced, $1\not\in \labl_{\edgi_1}^{-1}\, \labl_{\vertii}\, \labl_{\edgi_2} = \labl_{\bm{\walki}\pi}$ and by~\Cref{lem: wedge label}, $1\not\in \labl_{\bm{\walki} }=\labl_{\bm{\walki}\pi} \cap \labl_{\bm{\walki}\pi'}$.

It remains to show that $\recg{\junction{H}{K}} = H\cap K$. Indeed, let $\bm{\walki}$ be an arbitrary \bp-\walk\ in $\Atipr$, and let $\bm{\walki} =\boldsymbol{\walki_{1}}\cdots \boldsymbol{\walki_{r}}$ be its elementary decomposition. Clearly, the elementary decompositions of $\bm{\walki}\pi$ in $\Ati_{\!H} $, and $\bm{\walki}\pi'$ in $\Ati_{\!K} $, are $\bm{\walki}\pi =(\boldsymbol{\walki_1}\pi)\cdots (\boldsymbol{\walki_r}\pi)$ and $\bm{\walki}\pi' =(\boldsymbol{\walki_1}\pi')\cdots (\boldsymbol{\walki_r}\pi')$, respectively. Then, by~\Cref{lem: wedge label}, $\labl_{\bm{\walki}} =\labl_{\bm{\walki_{\!1}}}\cdots \labl_{\bm{\walki_{\!r}}} \subseteq \labl_{\bm{\walki_{\!1}}\pi}\cdots \labl_{\bm{\walki_{\!r}}\pi} =\labl_{(\bm{\walki_{\!1}}\pi) \cdots (\bm{\walki_{\!r}}\pi)} =\labl_{\bm{\walki}\pi} \subseteq H$. Since this is true for every $\bm{\walki}$, we deduce $\recg{\junction{H}{K}}\leqslant H$; and, by the symmetric argument, also ${\recg{\junction{H}{K}}\leqslant K}$.

For the other inclusion, take an element $g\in H\cap K$, and let $g=g_1 \cdots g_r$ be its syllable decomposition in $\Gpfi*\Gpfii$. Since $\Ati_{\!H} $ is a reduced automaton and $\recg{\Ati_H}=H$, Lemma~\ref{alternating} ensures us that $g\in \red{\labl}_{\bm{\walki}}$ for some \emph{alternating} $\bp$-\walk\ $\bm{\walki}$ from $\Ati_{\!H}$; in this situation, its elementary decomposition, $\bm{\walki} =\bm{\walki}_{\!1}\cdots \bm{\walki}_{\!r}$, corresponds to the syllable decomposition $g=g_1 \cdots g_r$, \ie $g_i \in \red{\labl}_{\bm{\walki}_{\!i}}$, for $i=1,\ldots ,r$. Similarly, there exists an \emph{alternating} $\bp$-\walk\ $\bm{\walki}'$ from $\Ati_{\!K}$, whose elementary decomposition $\bm{\walki}' =\bm{\walki}'_1\cdots \bm{\walki}'_r$ again corresponds to the syllable decomposition $g=g_1 \cdots g_r$, \ie $g_i \in \red{\labl}_{\bm{\walki}'_{\!i}}$, for $i=1,\ldots ,r$.

Write $\walki_{\!i}=\verti_{i-1}\edgi_i^{-1} \vertii_i \edgii_i \verti_i$ and $\walki'_{\!i}=\verti'_{i-1} ({\edgi_i'})^{-1} \vertii'_i \edgii'_i \verti'_i$.
Then, for each $i=1,\ldots ,r$, we have $g_i\in \labl_{\walki_{\!i}}=\labl_{\edgi_i}^{-1} \,\labl_{\vertii_i}\, \labl_{\edgii_i}$ and $g_i\in \labl_{\walki'_{\!i}}= \labl_{\edgi'_i}^{-1} \, \labl_{\vertii'_i} \, \labl_{\edgii'_i}$; so, $\varnothing \neq \labl_{\edgi_i}^{-1} \labl_{\vertii_i} \labl_{\edgii_i} \cap \labl_{\edgi'_i}^{-1} \labl_{\vertii'_i} \labl_{\edgii'_i} \subseteq G_{\nu}$. This means that $\bsf{\verti_{i-1}}=(\verti_{i-1}, \verti'_{i-1}) \equiv_{\nu_i} (\verti_i, \verti'_i)=\bsf{\verti_i}$\,, where $\nu_i$ is the common type of the vertices $\vertii_i$ (in $\Ati_{\!H} $) and $\vertii'_i$ (in $\Ati_{\!K} $). Therefore, $\bsf{\verti_{i-1}}$ and $\bsf{\verti_{i}}$ are both incident to a common $\nu_i$-secondary vertex in $\junction{H}{K}$. In other words, there is a $\nu_i$-elementary \walk\ in $\junction{H}{K}$, say $\boldsymbol{\walki_{\!i}}$, from~$\bsf{\verti_{i-1}}$ to $\bsf{\verti_{i}}$ . Finally, by~\Cref{lem: wedge label}, $g_i \in \labl_{\walki_{\!i}}\cap \labl_{\walki'_{\!i}} =\labl_{\bm{\walki}_{\!i}}$. Therefore, $g=g_1 \cdots g_r \in \labl_{\boldsymbol{\walki_1}}\cdots \labl_{\boldsymbol{\walki_r}} =\labl_{\boldsymbol{\walki_1} \cdots \boldsymbol{\walki_r}} \subseteq \recg{\junction{H}{K}}$, concluding the proof.
\end{proof}

\begin{cor}\label{cor: fg int <-> fg labels}
In the above situation, $H\cap K$ is finitely generated if and only if all the vertex labels of $\junction{H}{K}$ are finitely generated. \qed
\end{cor}


\subsection{Understanding intersections of cosets}

According to Lemma~\ref{lem: wedge automaton recognition}, given a wedge automaton $\Ati_{\!H}$, the union of labels of all the \walk s in $\Ati_{\!H} $ from the basepoint to a primary vertex $\verti \in \Vertexs_{0}\,\Ati_{\!H}$, denoted by $\recg{\Ati_{\!H}}_{(\sbp,\verti)}$, constitutes a coset of the recognized subgroup $\recg{\Ati_{\!H}} = H$. In general, though, this does not reflect all the cosets of $H$ (consider, for example the cases when $\Ati_{\!H}$ has only finitely many primary vertices, but $H$ has infinite index in $G_1 *G_2$). We can slightly modify the automaton~$\Ati_{\!H}$ to fix this issue.

Let $u=a_1 b_1\cdots a_s b_s\in \Gpfi*\Gpfii$, written in normal form. Consider the $(\Gpfi,\Gpfii)$-wedge automaton (also denoted by $u$) consisting on a chain spelling the normal form for $g$, and having trivial vertex label; let us call it the \defin{thread} for $u$. Attach this thread to $\Ati_{\!H} $ by identifying the basepoint $\bp_{H}$ with $\start u$, and then apply the folding process until no more foldings are possible (see the proof of~\Cref{thm: wedge reduced construction}). Observe that operation~(II) will not be used, and the triviality of the vertex labels in the thread implies that the vertex groups already present in $\Ati_{\!H} $ will not be changed along the process. So, the output is the exact same graph $\Ati_{\!H} $ with a terminal segment of the thread attached somewhere and sticking out; denote this new automaton by $\Ati_{Hu}$. Clearly, $\Ati_{Hu}$ is a reduced automaton, like $\Ati_{\!H} $, and furthermore, since the new secondary vertices out of $\Ati_{\!H} $ have trivial label, $\recg{\Ati_{Hu}}=\recg{\Ati_H}=H$.

By Lemma~\ref{lem: wedge automaton recognition}~(iii), $\recg{\Ati_{Hu}}_{(\sbp_H,\term u)} = \recg{\Ati_{\!H}} \cdot u$ (the situation where this coset could already be represented by a vertex in $\Ati_{\!H} $ corresponds to the fact that the thread happens to fold completely and so, $\Ati_{Hu}=\Ati_{\!H} $).

Now let us go back to the graph $\junction{H}{K}$. It is useful to understand the intersection of $H$ and $K$ but also, adding the corresponding hairs, it will be useful to understand the intersection of two arbitrary cosets $Hu$ and $Kv$.

Given elements $u,v\in \Gpfi*\Gpfii$, consider the reduced automata $\Ati_{Hu}$ and $\Ati_{Kv}$, and consider the junction automaton $\Ati_{\!Hu} \junct \Ati_{\!Kv}$.

\begin{lem}\label{lem: int cosets}
With the above notation,
\begin{enumerate}[ind]
\item $\junction{H}{K}$ is a reduced subgraph of $\Ati_{\!Hu} \junct \Ati_{\!Kv}$;
\item $Hu \cap Kv \neq \varnothing$ if and only if the vertex $(\term u, \term v)$ belongs to $\Ati_{\!Hu} \junct \Ati_{\!Kv}$;
\item for any \walk\ $\walki$ in $\Ati_{\!Hu} \junct \Ati_{\!Kv}$ from $(\bp_H,\bp_K)$ to $(\term u, \term v)$, and any $g\in \labl_{\walki}$, we have $Hu\cap Kv =\recg{\Ati_{\!Hu} \junct \Ati_{\!Kv}}_{(\sbp_H,\sbp_K), (\term u, \term v)}=(H\cap K)g$.
\end{enumerate}
\end{lem}

\begin{proof}
Note that the initial set of primary vertices for $\Ati_{\!Hu} \junct \Ati_{\!Kv}$, namely $\Vertexs_{0}\,\Ati_{Hu} \times \Vertexs_{0}\,\Ati_{Kv}$, contains as a subset $\Vertexs_{0}\,\Ati_{\!H} \times \Vertexs_{0}\,\Ati_{\!K} $, the initial set of primary vertices for $\junction{H}{K}$. And two old vertices $(\verti_1,\verti'_1),\, (\verti_2,\verti'_2)\in \junction{H}{K}$ are $\equiv_{\nu}$-equivalent in $\junction{H}{K}$ if and only if they are $\equiv_{\nu}$-equivalent as vertices in $\Ati_{\!Hu} \junct \Ati_{\!Kv}$ (since vertices of $\Ati_{\!Hu} \junct \Ati_{\!Kv}$ outside $\junction{H}{K}$ have always trivial labels). This proves (i).

Suppose first that $(\term u,\term v)\in \Vertexs\,\Ati_{\!Hu} \junct \Ati_{\!Kv}$, let $\bm{\walki}$ be a \walk\ in $\Ati_{\!Hu} \junct \Ati_{\!Kv}$ from $(\bp_H,\bp_K)$ to $(\term u,\term v)$, and consider its basic label $\labl^{\bullet}_{\bm{\walki}}\in \Gpfi*\Gpfii$. By the same argument as in Proposition~\ref{prop: junction recognizes itersection}, $\labl^{\bullet}_{\bm{\walki}}\in \labl_{\bm{\walki}\pi}\cap \labl_{\bm{\walki}\pi'}$. But $\bm{\walki}\pi$ (resp., $\bm{\walki}\pi'$) is a \walk\ in $\Ati_{Hu}$ from $\bp_H$ to $\term u$ (resp., a \walk\ in $\Ati_{Kv}$ from $\bp_K$ to $\term v$) hence, by Lemma~\ref{lem: wedge automaton recognition}~(iii), $\labl^{\bullet}_{\bm{\walki}} \in \recg{\Ati_{Hu}}_{(\sbp_H,\term u)} \cap \recg{\Ati_{Kv}}_{(\sbp_K,\term v)}=Hu \cap Kv$, concluding that $Hu\cap Kv \neq \varnothing$.

Conversely, suppose that $Hu\cap Kv \neq \varnothing$ and let $g \in Hu \cap Kv$.
Again by~\Cref{lem: wedge automaton recognition}~(iii), there exist \walk s $\walki$ in $\Ati_{Hu}$ from $\bp_H$ to $\term u$, and $\walki'$ in $\Ati_{Kv}$ from $\bp_K$ to $\term v$, such that $g\in \labl_{\walki} \cap \labl_{\walki'}$. Again, with an argument like in the proof of Proposition~\ref{prop: junction recognizes itersection}, there exists a \walk\  $\bm{\walki}$ in $\Ati_{\!Hu} \junct \Ati_{\!Kv}$ from $(\bp_H,\bp_K)$ to $(\term u, \term v)$ such that $g\in \labl_{\bm{\walki}}$. In particular, $(\term u, \term v)\in \Ati_{\!Hu} \junct \Ati_{\!Kv}$. This proves (ii) and~(iii).
\end{proof}

\subsection{Proofs of~\Cref{thm: ESIP-free,thm: TIP-free}}\label{sec: free product theorems}

Let us now address the algorithmic aspects of this construction. For all the present section, assume the two starting reduced $(\Gpfi, \Gpfii)$-automata $\Ati_{\!H} $ and $\Ati_{\!K} $ to be of finite type (namely, $H$ and $K$ are finitely generated subgroups of $\Gpfi*\Gpfii$). Recall that
although the underlying graph of $\junction{H}{K}$ is finite,
the labels of the vertices in $\junction{H}{K}$  may very well be non finitely generated as a result of intersections of finitely generated subgroups of $G_1$ and of~$G_2$.

A first easy observation is that, under the assumption that both $\Gpfi$ and $\Gpfii$ are Howson, then $\junction{H}{K}$ will always be of finite type. This proves that the free product of two Howson groups is again Howson, recovering a classical result originally proved by~\namecite{baumslag_intersections_1966}.



\begin{proof}[Proof of~\Cref{thm: ESIP-free}.] Assume that both $\Gpfi$ and $\Gpfii$ have solvable \ESIP; and suppose we are given two finitely generated subgroups $H,K\leqslant \Gpfi*\Gpfii$ by finite sets of generators, and two extra elements $u,v\in \Gpfi*\Gpfii$, all of them in normal form. By~\Cref{rem: CIPfg -> MP}, both $\Gpfi$ and $\Gpfii$ also have solvable membership problem; and by~\Cref{thm: wedge reduced construction}, we can compute reduced $(\Gpfi, \Gpfii)$-automata $\Ati_{\!H} $ and $\Ati_{\!K} $ such that $\recg{\Ati_{\!H}}=H$, and~$\recg{\Ati_{\!K}}=K$.

Now let us keep constructing $\junction{H}{K}$: we start looking at the basepoint $\bp= (\bp_H,\bp_K)$, with the whole set $\Vertexs_{0}\, \Atipr =\Vertexs_{0}\,\Ati_{\!H} \times \Vertexs_{0}\,\Ati_{\!K} $ in the background. We have to keep adding $\nu$-secondary vertices (with their labels), and $\nu$-arcs (with their labels too) connecting them to certain primaries, until getting $\junction{H}{K}$, the full connected component of~$\Atipr$ containing the basepoint~\bp.

We start checking whether there exists $\nu \in \set{1,2}$, such that both $\bp_H$ and $\bp_K$ have nonempty $\nu$-neighborhoods. If not, then the basepoint~\bp\ is not adjacent to any secondary vertex in $\Atipr$, and we are done (namely, the product
is the trivial automaton, and $H \cap K =\trivial$). Otherwise, let $\vertii\in \Vertexs_{\nu}\,\Ati_{\!H} $, and $\edgi\in \Edgi_{\nu}\,\Ati_{\!H} $ with $\start \edgi=\vertii$, $\term \edgi=\bp_H$; and let $\vertii'\in \Vertexs_{\nu}\,\Ati_{\!K} $ and $\edgi'\in \Edgi_{\nu}\,\Ati_{\!K} $ with $\start \edgi'=\vertii'$, $\term \edgi'=\bp_K$, and enlarge our picture by drawing a new $\nu$-secondary vertex, say~$\bsf{q}$, and a new $\nu$-arc, say $\bsf{e} = (\edgi,\edgi')$, from $\bsf{q}$ to $\bp$. According to~\eqref{label-c} --- and with respect to the choice $(\verti_{\bsf{q}}, \verti'_{\bsf{q}})=(\bp_H, \bp_K)$ --- we know that the label of $\bsf{q}$ is $\labl_{\bsf{q}}=\labl_\vertii^{\,\labl_{\edgi}} \cap \labl_{\vertii'}^{ \labl_{\edgi'}}\leqslant G_{\nu}$.

Applying \SIP\ for $G_{\nu}$ to the (finitely generated) subgroups $\labl_\vertii^{\,\labl_{\edgi}}$ and $\labl_{\vertii'}^{\,\labl_{\edgi'}}$, we can decide whether $\labl_{\bsf{q}}$ is finitely generated or not. In case it is not, kill the whole process and declare $H\cap K$ to be non finitely generated. Otherwise, compute a finite set of generators for $\labl_{\bsf{q}}$, assign $\labl_{\bsf{e}}=\trivial$, and check which other primary vertices from $\Atipr$ are adjacent to $\bsf{q}$: $\bsf{p}=(\verti,\verti')\in \Vertexs_{0}\,\Atipr$ is adjacent to $\bsf{q}$ if and only if $(\verti,\verti')\equiv_{\nu} (\verti_{\bsf{q}},\verti'_{\bsf{q}})$, which happens if and only if there exists $\edgii\in \Edgi_{\nu}\,\Ati_{\!H} $ from $\vertii$ to $\verti$, and $\edgii'\in \Edgi_{\nu}\,\Ati_{\!K} $ from $\vertii'$ to $\verti'$, such that~$\labl_\edgi^{-1}\labl_\vertii \, \labl_\edgii \,\cap\, \labl^{-1}_{\edgi'} \labl_{\vertii'}\labl_{\edgii'} \neq \varnothing$. So, run over every $\verti\in \Vertexs_{0}\Ati_{\!H} $ adjacent to $\vertii$, and every $\verti'\in \Vertexs_{0}\Ati_{\!K} $ adjacent to $\vertii'$ and, for each such pair, check whether the intersection of (right) cosets
 \begin{equation}\label{punt-important}
\labl_\vertii^{\,\labl_\edgi} \cdot (\labl_\edgi^{-1}\labl_\edgii) \,\cap\,
\labl_{\vertii'}^{\,\labl_{\edgi'}} \cdot (\labl^{-1}_{\edgi'} \labl_{\edgii'})
= \labl_\edgi^{-1}\labl_\vertii\,\labl_\edgii \,\cap\, \labl^{-1}_{\edgi'}\labl_{\vertii'}\,\labl_{\edgii'}
 \end{equation}
is empty or not; this can be done using the above call to \ESIP\ from $G_{\nu}$, since they are right cosets of $\labl_\vertii^{\,\labl_\edgi}, \labl_{\vertii'}^{\, \labl_{\edgi'}}\leqslant G_{\nu}$, whose intersection happens to be finitely generated.

In case this intersection is not empty, add a $\nu$-arc, say $\bsf{f} = (\edgii,\edgii')$, from $\bsf{q}$ to $\bsf{p}$, and $\labl_{(\edgii,\edgii')}$ arbitrarily chosen from that nonempty intersection. After this procedure, we have a complete picture of the 1-elementary and 2-elementary \walk s in $\junction{H}{K}$ starting at the basepoint \bp.

Now, for  every $\nu=1,2$, and every primary vertex $\bsf{p} = (\verti,\verti')$ added to the picture and not yet explored, repeat the same process (with $\bsf{p}$ in place of $\bp$). Since the underlying graph of $\Atipr$ is finite, this procedure will either find a non finitely generated vertex label, or will finish the complete construction of $\junction{H}{K}$ in finite time. In the first case, we deduce the non finitely generated type of $H\cap K$; in the second case, we can compute generators for $H\cap K$ (in fact, a Kurosh decomposition) applying~\Cref{thm: from reduced to Kurosh}.

Hence, so far, we have solved $\SIP(\Gpfi *\Gpfii)$. To finish the proof, let us place ourselves in the case where $H\cap K$ is finitely generated (and so, with the junction automaton $\junction{H}{K}$ fully constructed), and let us decide whether the intersection of right cosets $Hu\cap Kv$ is empty or not. We can extend the computation of $\junction{H}{K}$ to that of $\Ati_{\!Hu} \junct \Ati_{\!Kv}$; or alternatively construct directly $\Ati_{\!Hu} \junct \Ati_{\!Kv}$ from the beginning.

It only remains to check whether the vertex $(\term u,\term v)$ appears in~$\Ati_{\!Hu} \junct \Ati_{\!Kv}$, or not. Namely, using \Cref{lem: int cosets}~(ii): $(\term u,\term v)$ is connected to $(\bp_H, \bp_K)$ if and only if the intersection $Hu\cap Kv$ is nonempty; and, if so, any element $g$ from the label of any \walk\ from $(\bp_H, \bp_K)$ to $(\term u,\term v)$ belongs to such intersection, $g\in Hu\cap Kv=(H\cap K)g$. This concludes the proof. \qed

Finally, we complement the arguments in the last proof to prove that \TIP\ also passes through free products.

\emph{Proof of~\Cref{thm: TIP-free}.} Since solvability of \TIP\ implies that of \ESIP, \Cref{thm: ESIP-free} already gives us \ESIP\ for $\Gpfi*\Gpfii$. It remains to solve \CIP\ for $\Gpfi*\Gpfii$ in the case where the given finitely generated subgroups $H,K$ have a non finitely generated intersection.

Given $H,K\leqslant \Gpfi*\Gpfii$ finitely generated, and $u,v\in \Gpfi*\Gpfii$, run the same algorithm as in the proof of~\Cref{thm: ESIP-free}: construct $\Ati_{Hu}$ and $\Ati_{Kv}$ and start building the junction $\Ati_{\!Hu} \junct \Ati_{\!Kv}$; when we encounter a secondary vertex $\bsf{q}$ whose label $\labl_{\bsf{q}} =\labl_\vertii^{ \labl_{\edgi}} \cap \labl_{\vertii'}^{\,\labl_{\edgi'}}\leqslant G_{\nu}$ is not finitely generated, instead of computing a set of generators for it (which is not possible), we just put the trivial subgroup as a label in place of $\labl_{\bsf{q}}$. Then, when analyzing which other primary vertices are adjacent to $\bsf{q}$, we need to decide if the intersection of cosets from equation~\eqref{punt-important} are empty or not: even though  $\labl_\vertii^{\,\labl_{\edgi}} \cap \labl_{\vertii'}^{\,\labl_{\edgi'}}\leqslant G_{\nu}$ is not finitely generated, the decision can be made effective using \CIP\ from $G_{\nu}$. This way, we can algorithmically complete the description of $\Ati_{\!Hu} \junct \Ati_{\!Kv}$ except that, for some secondary vertices $\bsf{q}$, instead of having generators for $\labl_{\bsf{q}}$, we just have the trivial element labelling them.

Of course, this is not enough information for computing a set of generators for $H\cap K$. But it suffices for deciding whether the vertices $(\bp_H,\bp_K)$ and $(\term w, \term w')$ belong to the same connected component of $\Ati_{\!Hu} \junct \Ati_{\!Kv}$. By Lemma~\ref{lem: int cosets}, this allows us to decide whether the intersection of cosets $Hu\cap Kv$ is empty or not; and in case it is not, we can compute an element from it, just choosing a \walk\ $\walki$ from $(\bp_H,\bp_K)$ to $(\term w, \term w')$, and then picking an element from $\labl_{\walki}$ (if $\walki$ traverses some secondary vertex with a non finitely generated label, we just recorded the trivial element from it for this purpose). This completes the proof.
\end{proof}

\subsection*{Acknowledgments}
The first two named authors acknowledge financial support from the Spanish Agencia Estatal de Investigaci\'on, through grant MTM2017-82740-P (AEI/FEDER, UE), and also the ``María de Maeztu'' Programme for Units of Excellence in R\&D (MDM-2014-0445). The third named author acknowledges support by CMUP (UID/MAT/00144/2013), which is funded by FCT (Portugal) with national (MEC) and European structural funds (FEDER), under the partnership agreement PT2020; he was also partially supported by the ERC Grant 336983, by the Basque Government grant IT974-16, by the grant MTM2014-53810-C2-2-P of the Ministerio de Economia y Competitividad of Spain, and by the Russian Foundation for Basic Research (project no. 15-01-05823).

\renewcommand{\bibfont}{\small}
\printbibliography

\end{document}